\documentclass[12pt,reqno]{amsart}
\usepackage{mathrsfs}
\usepackage[breaklinks]{hyperref}
\usepackage[dvipsnames]{xcolor} 
\usepackage{diagbox}
\usepackage{latexsym}
\usepackage{upref, eucal}
\usepackage{tikz}
\usepackage{amsmath, amssymb, graphicx}
\usepackage{epsfig}
\usepackage{tkz-euclide}

\usepackage[headheight=110pt,top=1in, bottom=.9in, left=0.8in, right=0.8in]{geometry}
\allowdisplaybreaks
\usepackage{url}
\usepackage{amssymb}

\renewcommand{\(}{\left\(}
\renewcommand{\)}{\right\)}
\renewcommand{\[}{\left\[}
\renewcommand{\]}{\right\]}
 
\newcommand{\qbinom}[2]{\genfrac{[}{]}{0pt}{}{#1}{#2}}
\renewcommand{\i}{\infty}
\renewcommand{\pmod}[1]{\,(\textup{mod}\,#1)}
\numberwithin{equation}{section}
\theoremstyle{plain}
\newtheorem{theorem}{Theorem}[section]
\newtheorem{lemma}[theorem]{Lemma}
\newtheorem{corollary}[theorem]{Corollary}

\newtheorem{remark}[]{Remark}

\makeatletter
\def\proof{\@ifnextchar[{\@oproof}{\@nproof}}
\def\@oproof[#1][#2]{\trivlist\item[\hskip\labelsep\textit{#2 Proof of\
		#1.}~]\ignorespaces}
\def\@nproof{\trivlist\item[\hskip\labelsep\textit{Proof.}~]\ignorespaces}

\makeatother

\makeatletter
\def\@tocline#1#2#3#4#5#6#7{\relax
	\ifnum #1>\c@tocdepth 
	\else
	\par \addpenalty\@secpenalty\addvspace{#2}%
	\begingroup \hyphenpenalty\@M
	\@ifempty{#4}{%
		\@tempdima\csname r@tocindent\number#1\endcsname\relax
	}{%
		\@tempdima#4\relax
	}%
	\parindent\z@ \leftskip#3\relax \advance\leftskip\@tempdima\relax
	\rightskip\@pnumwidth plus4em \parfillskip-\@pnumwidth
	#5\leavevmode\hskip-\@tempdima
	\ifcase #1
	\or\or \hskip 1em \or \hskip 2em \else \hskip 3em \fi%
	#6\nobreak\relax
	\dotfill\hbox to\@pnumwidth{\@tocpagenum{#7}}\par
	\nobreak
	\endgroup
	\fi}
\makeatother

\allowdisplaybreaks

\begin{document}
	
\title[Non-Rascoe partitions and a rank parity function]{Non-Rascoe partitions and a rank parity function associated to the Rogers--Ramanujan partitions}
\author{Atul Dixit, Gaurav Kumar and Aviral Srivastava}
\address{Department of Mathematics, Indian Institute of Technology Gandhinagar, Palaj, Gandhinagar 382355, Gujarat, India} 
\email{adixit@iitgn.ac.in; kumargaurav@iitgn.ac.in; aviral.srivastava@iitgn.ac.in}
\begin{abstract}
We study the generating function of the excess number of Rogers--Ramanujan partitions with odd rank over those with even rank, and, using combinatorial and analytical techniques, show that this generating function is closely connected with an interesting class of restricted partitions, namely, partitions into distinct parts where the number of parts is not a part. We derive arithmetic properties of the number of such partitions and conjecture an interesting mod $4$ congruence. Generalizations of most of these results in a parameter $\ell$ are also obtained.    
\end{abstract}
\maketitle
\tableofcontents

\section{Introduction}\label{intro}

A \emph{partition} of a positive integer $n$ is a non-increasing sequence of positive integers whose sum is $n$, with each positive integer forming the partition termed as a \emph{part} of that partition. Dyson's \emph{rank} of a partition is defined as the  largest part of the partition minus the number of parts. For example, the rank of the partition $4+1+1$ of $6$ is $4-3=1$.  

The enumeration of partitions, belonging to a particular class, based on their rank parity is often fruitful since the generating functions involved turn out to be important $q$-series with interesting modular behavior. For instance, the generating function of the excess number of partitions with even rank over those with odd rank is Ramanujan's third order mock theta function $f(q):=\displaystyle\sum_{n=0}^{\infty}\frac{q^{n^2}}{(-q;q)_n^2}$, where, here, and throughout the paper, the following standard $q$-series notation will be adopted:
\begin{align*}
	(a)_0 &:=(a;q)_0 =1, \qquad \\
	(a)_n &:=(a;q)_n  = (1-a)(1-aq)\cdots(1-aq^{n-1}),
	\qquad n \geq 1, \\
	(a)_{\infty} &:=(a;q)_{\i}  = \lim_{n\to\i}(a;q)_n, \qquad |q|<1,\\
	(a)_{-n}&:=\left(1-a/q^n\right)^{-1}\left(1-a/q^{n-1}\right)^{-1}\cdots\left(1-a/q\right)^{-1},\quad n \geq 1,
\end{align*}
with $q$ denoting the base. We will always consider $|q|<1$ in this paper.

The function $f(q)$ is essentially a weight $1/2$ mock modular form. Choi, Kang and Lovejoy \cite{choi-kang-lovejoy} considered the corresponding problem with rank replaced by the Andrews-Garvan crank.

On the other hand, if we consider only the partitions into distinct parts, that is, partitions whose parts differ by at least $1$, then the excess number of such partitions with even rank over those with odd rank is another famous function of Ramanujan \cite{andrews1986} , namely,
\begin{equation*} \sigma(q):=\sum_{n=0}^{\infty}\frac{q^{n(n+1)/2}}{(-q)_n}.
\end{equation*}
This function was first considered by Ramanujan in his Lost Notebook in conjunction with two `sum-of-tails identities' \cite[p.~14]{lnb}. See \cite{andrews1986} for details. It is a prototypical example of a quantum modular form \cite{zagierqmf}.  It has connections to algebraic number theory since its coefficients have properties governed by the arithmetic of $Q(\sqrt{6})$. It was this connection which enabled Andrews, Dyson and Hickerson \cite{adh} to prove two mysterious properties of $\sigma(q)$ first conjectured by Andrews \cite{andrewsmonthly86}. Cohen \cite{cohen} proved that if $	\varphi(q):=q^{1/24}\sigma(q)+q^{-1/24}\sigma^{*}(q)=\sum_{n\in\mathbb{Z}\atop n\equiv1\hspace{1mm}(\text{mod}\hspace{1mm}24)}T(n)q^{|n|/24}$, where $\sigma^{*}(q):=2\sum_{n=1}^{\infty}(-1)^nq^{n^2}/(q;q^2)_n$, then $T(n)$ are the coefficients of a Maass waveform of eigenvalue $1/4$. The literature on $\sigma(q)$ is vast. The reader is referred to \cite{bandix1} for an extensive bibliography.

Now consider partitions in which parts differ by at least $2$. These are called the \emph{Rogers--Ramanujan partitions}\footnote{They are sometimes called Rogers--Ramanujan partitions of type I.}, for, they are generated by the ``sum side'' of the first Rogers--Ramanujan identity \cite{rogers2}, namely,
\begin{align}\label{rr1}
\sum_{n=0}^{\infty}\frac{q^{n^2}}{(q)_n}=\frac{1}{(q;q^5)_{\infty}(q^4;q^5)_{\infty}}.
\end{align}
Here, we are interested in the excess number of Rogers--Ramanujan partitions with even rank over those with odd rank. To obtain its generating function, define $R(m, n)$ to be the number of  Rogers--Ramanujan partitions of $n$ with rank $m$. Our first result is 
\begin{theorem}\label{theoremrrp}
\begin{align*}
	\sum_{n=0}^{\infty} \sum_{m=-\infty}^{\infty} R(m,n) z^m q^n = \sum_{n=0}^{\infty} \frac{z^{n-1}q^{n^2}}{(zq)_n}.
	\end{align*}
\end{theorem} 
Thus, letting $z=-1$ in the above result, we find that the generating function of the excess number of Rogers--Ramanujan partitions with \emph{odd} rank over those with \emph{even} rank is the function
\begin{align}\label{sigma2(q) defn}
	\sigma_2(q):=\sum_{n=0}^{\infty} \frac{(-1)^nq^{n^2}}{(-q)_n}=1-q+q^2-q^3+2q^4-2q^5+q^6-q^7+2q^8-3q^9+\cdots.
\end{align}
We call $\sigma_2(q)$ as a \emph{rank parity function associated to the Rogers--Ramanujan partitions}.

Let us illustrate this special case $z=-1$  of Theorem \ref{theoremrrp} with an example. Let $n=9$. The only Rogers--Ramanujan partition with odd rank is $7+2$ whereas the Rogers--Ramanujan partition with even rank are $9, 8+1, 6+3$ and $5+3+1$. Hence the excess number of Rogers--Ramanujan partitions with odd rank over those with even rank is $1-4=-3$, which agrees with the coefficient of $q^{9}$ in \eqref{sigma2(q) defn}.
 
 Had the $(-1)^n$ in the summand of the series in \eqref{sigma2(q) defn} been absent, the above sum would have been Ramanujan's fifth order mock theta function $f_0(q)$, which has been extensively analyzed from the perspective of mock modular forms, $q$-series and partitions. On the other hand, the function $\sigma_2(q)$ seems to have been conspicuously missing from the literature. One of the goals of this paper is to undertake a comprehensive study of this function from combinatorial as well as analytical point of view. It will be seen here that $\sigma_2(q)$ enjoys several interesting properties and so it will not be surprising if it eventually ends up being in the same league as $\sigma(q)$ in the future.

Before discussing the combinatorics associated with $\sigma_2(q)$ though, we  introduce new restricted partitions introduced in 2024 by John Tyler Rascoe \cite{jtr1}, which we term \emph{Rascoe partitions}. These are the partitions of a positive integer into distinct parts in which the number of parts is itself a part of the partition. For example, if $n=11$, then the admissible partitions are $9+2$, $7+3+1$ and $ 6+3+2$. Let $a(n)$ denote the number of Rascoe partitions of $n$. Clearly, $a(0)=0$. Rascoe stated \cite{jtr1}, without proof, that
	\begin{align*}
	\sum_{n=1}^{\infty} a(n) q^n = \sum_{n=1}^{\infty} q^{n(n+1)/2}(q)_{n-1} \sum_{m=1}^{n} \frac{q^{m(m-1)}}{(q)_{n-m}(q)_{m-1}^2}.
\end{align*}
Next, we define a \emph{non-Rascoe partition} of a positive integer $n$  to be a partition of $n$ into distinct parts in which the number of parts is \emph{not} a part of the partition. For example, if $n=11$, then the non-Rascoe partitions of $11$ are  $11, 10+1, 8+3, 8+2+1, 7+4, 6+5, 6+4+1, 5+4+2$ and $5+3+2+1$. Let $b(n)$ denote the number of non-Rascoe partitions of $n$. We take $b(0)=1$ because the empty partition has $0$ parts which is, of course, not its part.

The sequence $\{b(n)\}_{n=1}^{\infty}$ has been tabulated in the OEIS \cite{jtr2}. Clearly,
 \begin{align*}
 \sum_{n=0}^{\infty} b(n) q^n =(-q)_{\infty}-\sum_{n=1}^{\infty} q^{n(n+1)/2}(q)_{n-1} \sum_{m=1}^{n} \frac{q^{m(m-1)}}{(q)_{n-m}(q)_{m-1}^2}.
\end{align*}
Our second result gives a surprising new connection between the generating function of $b(n)$ and $\sigma_2(q)$ as shown in the next theorem.
\begin{theorem}\label{theorem main}
\begin{align*}
\sum_{n=0}^{\infty} b(n) q^n =(-q)_{\infty}\sigma_2(q)=\sum_{n=0}^{\infty}(-1)^nq^{n^2}(-q^{n+1})_{\infty}.
\end{align*}
In other words, the non-Rascoe partitions are generated by $(-q)_{\infty} \sigma_2(q)$.
\end{theorem}
One may wonder what do we get if we perform the same exercise as above but without the restriction that the parts of a partition be distinct. Indeed, as will be shown in Section \ref{analogues}, the generating function of the number of partitions in which the number of parts is not a part is simply $(1-q+q^2)/(q)_{\infty}$. 

Two applications of Theorem \ref{theorem main} are given in the following two theorems.


 \begin{theorem}\label{non-rascoe odd}
	The number of non-Rascoe partitions of a number $n$ is odd if and only if $n=m(5m+1)/2$, where $m\in\mathbb{Z}$.
\end{theorem}

\begin{theorem}\label{p(i,n)}
	Let $P(j,n)$ be the number of partitions of $n$ into distinct parts with the exception that the smallest part, say $j$, is allowed to repeat exactly $j$ number of times. Then the excess number of such partitions with even smallest part over those with odd smallest part equals the number of non-Rascoe partitions of $n$, that is,
	\begin{equation*}
		\sum_{j=0}^{n}(-1)^jP(j, n)=b(n),
	\end{equation*}
	where, $P(0, n)$ is, clearly, the number of partitions of $n$ into distinct parts.
\end{theorem}
	We now illustrate the above theorem by means of an example. Let $n=11$.   As demonstrated after defining non-Rascoe partitions, $b(11)=9$. On the other hand, since $P(j, n)$ denotes the number of partitions into distinct parts with the exception that the frequency (number of appearances) of the smallest part $j$ is equal to $j$, we have $P(0,11)=12$, namely, the number of partitions of $11$ into distinct parts. Moreover, $P(1,11)=5$ as the number of admissible partitions are $10+1, 8+2+1, 7+3+1, 6+4+1$ and $5+3+2+1$. Also, $P(2, 11)=2$, since only $7+2+2$ and $4+3+2+2$ are admissible. Then $P(0,11)-P(1,11)+P(2,11)=12-5+2=9$, which agrees with $b(11)$.
	 
We note in passing that Andrews and El Bachraoui \cite[Theorem 1]{andrews-bachraoui jmaa} have recently studied partitions where the smallest part appears exactly $k$ times, where $k\in\mathbb{N}$ is fixed, and the remaining parts are distinct. They have shown that the generating function of the number of such partitions is a linear combination of $q$-shifted factorials with polynomials in the variable $q$ as coefficients.\\
 
 The parity of $b(n)$, that is, the number of non-Rascoe partitions of $n$, was precisely found in Theorem \ref{non-rascoe odd}. One may wonder if  there exist congruences that $b(n)$ satisfies over certain arithmetic progressions. We conducted a numerical search for such congruences up to moduli $\leq 31$ and found only one congruence. It is a rather exotic one. We are unable to prove it as of now. Hence we state it as a conjecture.\\
  
 \noindent
 \textbf{Conjecture 1:} Let $b(n)$ denote the number of non-Rascoe partitions of an integer $n$. For $k\geq1$ and $k$ not a multiple of $29$, the following congruence holds:
 \begin{align*}
 b(29k+21)\equiv0\pmod{4}.
 \end{align*}
 \begin{remark}
 	This conjecture has been verified for  the first $10^5$ values in $\{b(n)\}_{n=0}^{\infty}$.
 \end{remark}
\begin{remark}\label{exceptions}
If $k$ is a multiple of $29$, say, $k=29m$, then $b(29k+21)$ may or may not be divisible by $4$. For example, for $0\leq m\leq 118$, $b(29\cdot29m+21)$ is divisible by $4$ when $m=0, 1, 7, 8, 13, 19, 22, 27, 28, 29, 32, 37, 41, 44, 47, 48, 49, 50, 51, 52, 53, 57, 64, 67, 69, 74, 75, 76, 77, 78, 79,\\ 81, 82, 83, 84, 85, 89, 95, 100, 102, 104, 106, 108, 109, 115, 116, 117$ and $118$, and is not divisible by $4$ for the remaining values of $m$.
\end{remark}
 
 \begin{remark}
If Conjecture $1$ is valid, then, along with Theorem \ref{non-rascoe odd}, it implies that a number of the form $29n+21$, where $n\neq 29j$ can never be written in the form $m(5m+1)/2$ for any $m\in\mathbb{Z}$. This, can, of course, be proved using elementary number theory. 
 \end{remark}	
 
One of the important features of Ramanujan's rank parity function $\sigma(q)$ is that it has a Hecke-Rogers type representation, namely \cite[Theorem 1]{adh},
\begin{align*}
\sigma(q)=\sum_{n\geq 0\atop |j|\leq n }^{\infty}(-1)^{n+j}q^{n(3n+1)/2-j^2}(1-q^{2n+1}).
\end{align*}
A general $q$-series identity obtained by Liu \cite[Theorem 1.9]{Liu} not only gives, as a special case,  an identity \cite[Corollary 6.1]{Liu} which is a source of obtaining Hecke-Rogers type representations for many $q$-series, but also the  representation (i) for $\sigma_2(q)$ given in the following theorem. Moreover, an identity of Andrews \cite[Theorem 18, Equation (12.3)]{AndPPI} involving the little $q$-Jacobi polynomials results in the second representation for $\sigma_2(q)$.
 \begin{theorem} \label{hecke-type}
 	\begin{align*}
 		\textup{(i)}\hspace{5mm} \sigma_2(q)&=\frac{1}{(-q)_{\infty}} \bigg\{ 1+ \sum_{n=1}^{\infty}\frac{(-1)^nq^{n(5n-1)/2}(1+q^{2n})}{(1+q^n)} \sum_{j=0}^{n} \frac{(-1)_j}{(q)_j}\left(-q^{1-n}\right)^j \bigg\},\\
 		\textup{(ii)}\hspace{5mm} \sigma_2(q)&=\frac{1}{(-q)_{\infty}} \bigg\{ 1-\sum_{n=1}^{\infty} \frac{(-1)^nq^{n(5n-1)/2}(1+q^{2n})}{(1-q^n)} \sum_{m=0}^{n-1} \frac{(-1)_m}{(q)_m} (-q^{-n})^m \bigg\}. 
 	\end{align*}
 \end{theorem}
The two representations of $\sigma_2(q)$ given above look quite similar with only a few differences here and there. Thus, one may wonder if we can derive one from the other. Indeed, in Section \ref{analytic}, we prove that
	\begin{align}\label{theoremfinite1}
	\sum_{j=0}^{n} \frac{(-1)_j(-1)^jq^{j-jn}}{(q)_j}=-\frac{(1+q^n)}{(1-q^n)} \sum_{m=0}^{n-1} \frac{(-1)_m(-1)^mq^{-mn}}{(q)_m} .
\end{align}
 In Section \ref{gen rascoe}, we will generalize most of the above results in conjunction with a generalization of $\sigma_2(q)$ as well as with the number of \emph{generalized Rascoe partitions}, that is, $a_\ell(n)$, and the number of \emph{generalized non-Rascoe partitions} $b_\ell(n)$, where $\ell\in\mathbb{N}\cup\{0\}$.  A generalized Rascoe partition of an integer $n$ is a partition into distinct parts such that the number of parts plus $\ell$ is also a part. Similarly, a generalized non-Rascoe partition is a partition into distinct parts such that the number of parts plus $\ell$ is \emph{not} a part. In the same section, the following surprising generalization of \eqref{theoremfinite1} will be proved.
 \begin{theorem}\label{theoremgfinite1}
 	For $n, \ell\in\mathbb{Z}$,
 	\begin{align}\label{general ell finite identity}
 		\sum_{j=0}^{n} \frac{(-1)_j(-1)^jq^{-j(n+\ell)}}{(q)_j}=(-1)^{\ell}\frac{(q^{n+1})_{\ell}}{(-q^{n+1})_{\ell}} \sum_{j=0}^{n+\ell} \frac{(-1)_j(-1)^jq^{-jn}}{(q)_j} .
 	\end{align}
 \end{theorem}
 The reason why \eqref{general ell finite identity} is interesting is because, it is an identity between two finite analogues of Fine's function $F(a,b;t):=\sum_{n=0}^{\infty}(aq)_nt^n/(bq)_n$ where the number of terms summed on both sides are arbitrary. Also, it can be put in a completely symmetric form in $n$ and $n+\ell$, namely,
 \begin{align*}
 	(-1)^n\frac{(q)_n}{(-q)_n}\sum_{j=0}^{n} \frac{(-1)_j(-1)^jq^{-j(n+\ell)}}{(q)_j}=(-1)^{n+\ell}\frac{(q)_{n+\ell}}{(-q)_{n+\ell}} \sum_{j=0}^{n+\ell} \frac{(-1)_j(-1)^jq^{-jn}}{(q)_j} .
 \end{align*} 
In Section \ref{tenth}, we offer an identity linking the generating functions of $b_\ell(n)$ and $b_{\ell+1}(n)$ by making use of a mysterious identity of Ramanujan \cite{lnb}. A congruence involving $b_0(n)(=b(n)), b_1(n)$ and the coefficients of Ramanujan's tenth order mock theta functions $X(q)$ and $\chi(q)$ is given. Similarly, a congruence involving $b(n)$, $b_1(n)$ and certain  Fourier coefficients of the Hauptmoduln $j_5(\tau)$ is obtained using some recent results of Kang, Kim, Matsusaka and Yoo \cite{matsusaka}.

\section{Combinatorics of $\sigma_2(q)$ and non-Rascoe partitions}\label{comb}

In this section, the deep connection between $\sigma_2(q)$ and $b(n)$, the number of non-Rascoe partitions of $n$, given in Theorem \ref{theorem main} will be revealed. However, this requires proving several preparatory results, which are of independent interest.

We begin with obtaining an expression for the generating function of $R(m, n)$.
\begin{proof}[Theorem \textup{\ref{theoremrrp}}][]
	Let $n$ be the number of parts and $k$ be the largest part of a Rogers--Ramanujan partition. Since the difference between any two consecutive parts is at least $2$, we separate out the first $2j-1$ nodes from $j$-th part of the partition counted from the bottom, where $1\leq j\leq n$ as shown by the dotted region in Figure \ref{fig1}. This gives us $1+3+\cdots+(2n-1)=n^2$ nodes. Next, $1/(zq)_n$ generates the partition to the right of the dotted region with $z$ keeping track of $k-(2n-1)$. Thus, $\frac{z^{n-1}q^{n^2}}{(zq)_n}$ generates a Rogers--Ramanujan partition where the number of parts is $n$ and $z$ keeps track of the rank $k-n$.  Summing over $n$ from $0$ to $\infty$ leads us to the desired result.

\begin{figure}[hbt!]
	\begin{center}
		\includegraphics[width=0.40\textwidth]{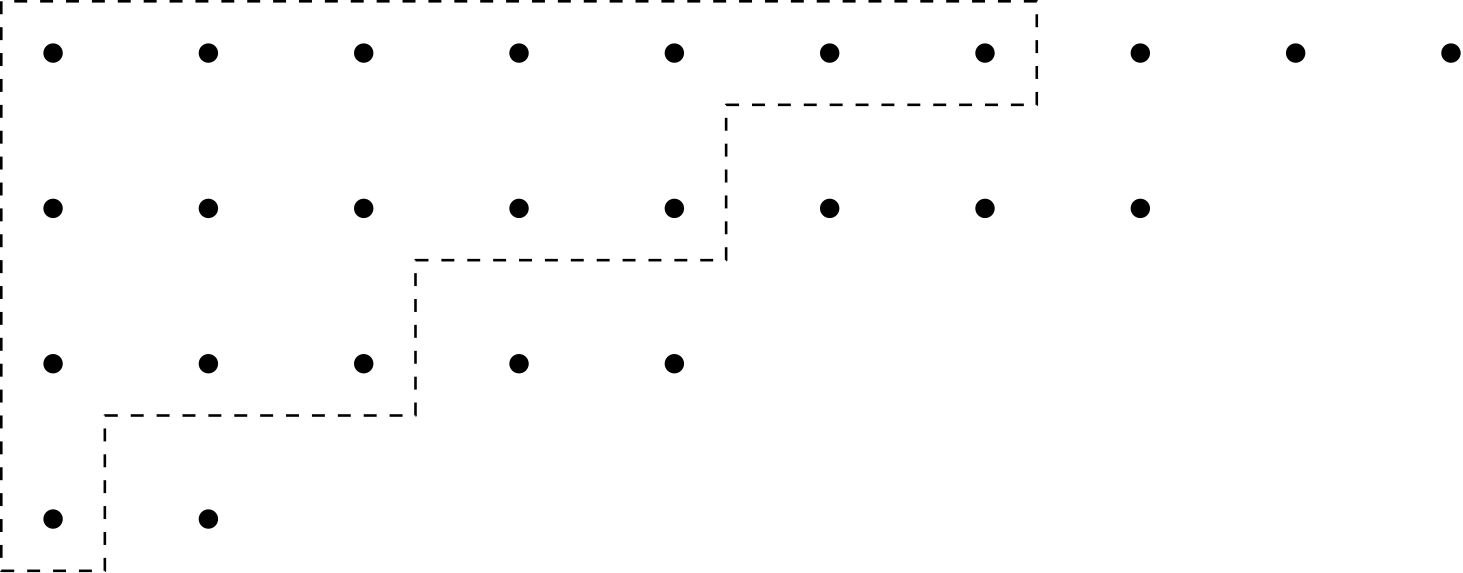}
		\caption{A Rogers--Ramanujan partition}
		\label{fig1}
	\end{center}
\end{figure}

\end{proof}

\begin{remark}
This result can also be derived using MacMahon's $\Omega$-operator method by starting with the representation
\begin{equation*}
\sum_{n=0}^{\infty} \sum_{m=-\infty}^{\infty} R(m,n) z^m q^n=1+\sum_{j=1}^{\infty}\sum_{n_1=1}^{\infty}\sum_{n_2=n_1+2}^{\infty}\cdots\sum_{n_{j}=n_{j-1}+2}^{\infty}z^{n_j-j}q^{n_1+n_2+\cdots+n_j}.
\end{equation*}
\end{remark}
Next, we connect $R(m, n)$ with another restricted partition function. 
	\begin{theorem}
	Let $L(m,N)$ be the number of partitions of a number $N$ into $m$ parts, with largest part, say $n$, repeating at least $n$ times. Then, $L(m,N)$ equals the Rogers--Ramanujan partitions of $N$ with rank $m-1$. In other words, $L(m,N)=R(m-1,N)$.
\end{theorem}
\begin{proof}
In view of Theorem \ref{theoremrrp}, it suffices to show that the generating function of $L(m,N)$ is
\begin{align*}
	\sum_{N=0}^{\infty} \sum_{m=0}^{\infty} L(m,N) z^m q^N = \sum_{n=0}^{\infty} \frac{z^nq^{n^2}}{(zq)_n}.
\end{align*}
\begin{figure}[hbt!]
	\begin{center}
		\includegraphics[width=0.13\textwidth]{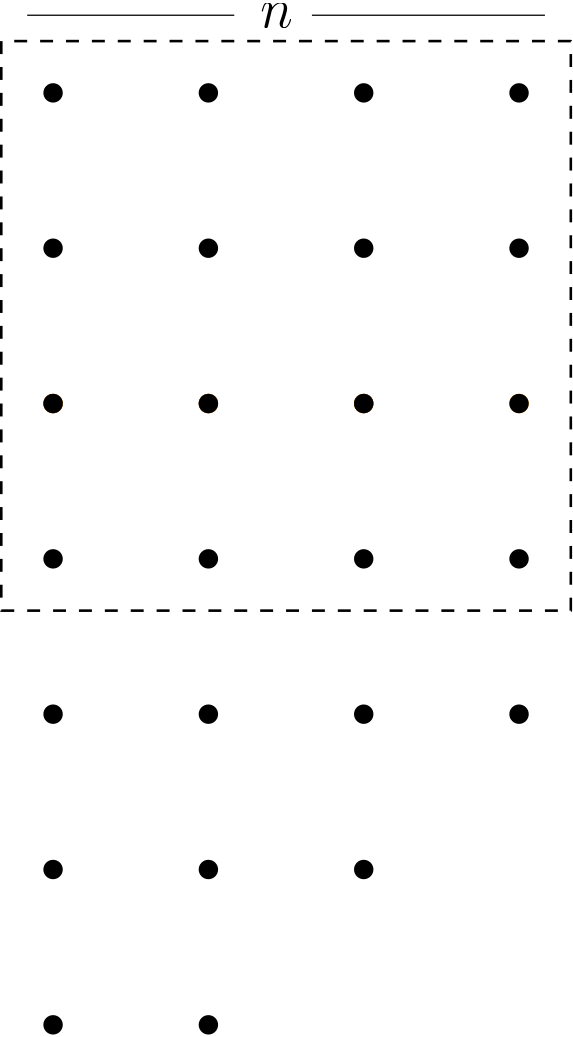}
		\caption{A partition generated by $L(m,N)$}
		\label{fig2}
	\end{center}
\end{figure}
From Figure \ref{fig2}, it is clear that if $n$ is the largest part of a partition enumerated by $L(m, N)$, then the dotted region is generated by $q^{n^2}$. Also, $1/(zq)_n$ generates the partition below the dotted region with $z$ keeping track of the difference between the total number of parts and $n$. Thus, $\frac{z^{n}q^{n^2}}{(zq)_n}$ generates a partition enumerated by $L(m,N)$ having $n$ as its largest part and $z$ keeping track of the number of parts. The desired generating function is then obtained by summing over $n$ from $0$ to $\infty$.
\end{proof}

\begin{remark}
It is easy to see that the partitions in Figures \ref{fig1} and \ref{fig2} can be bijectively mapped to each other.
\end{remark}
We now give representations for the generating functions of $a(N)$ and $b(N)$. Here, and throughout the sequel, we denote the Gaussian polynomial or the $q$-binomial coefficient by \cite[p.~35]{gea}
\begin{align*}
	\left[\begin{matrix} N\\n\end{matrix}\right]=\left[\begin{matrix} N\\n\end{matrix}\right]_q :=\begin{cases}
		\frac{(q;q)_N}{(q;q)_n (q;q)_{N-n}},\hspace{2mm}\text{if}\hspace{1mm}0\leq n\leq N,\\
		0,\hspace{2mm}\text{otherwise}.
	\end{cases} 
\end{align*}
\begin{theorem}\label{theoremrp-thm-1}
	\begin{align}
		\textup{(i)} \sum_{N=1}^{\infty} a(N) q^N &= \sum_{n=1}^{\infty} q^{n(n+1)/2} \sum_{m=1}^{n} \qbinom{n-1}{m-1}\frac{q^{m(m-1)}}{(q)_{m-1}},\label{theoremrp}\\
		\textup{(ii)} 	\sum_{N=1}^{\infty} b(N) q^N &= \sum_{n=1}^{\infty} q^{n(n+1)/2} \sum_{m=1}^{n} \qbinom{n-1}{m-1} \frac{q^{m^2}}{(q)_m}.\label{theoremnrp}
	\end{align}
\end{theorem}
\begin{proof}
In proving both (i) and (ii),we assume $n$ to be the number of parts in a partition, and $m$ to be the number of parts greater than $n$. 

We begin by proving (i). Clearly, $0\leq m\leq n-1$. The Ferrers diagram of a typical Rascoe partition is given in Figure \ref{figure3}, where the number of parts, in this case, $6$, is also a part. We divide the Ferrers diagram into four different regions - $A, B, C$ and $D$. 
		\begin{figure}[hbt!]
		\begin{center}
			\includegraphics[width=0.50\textwidth]{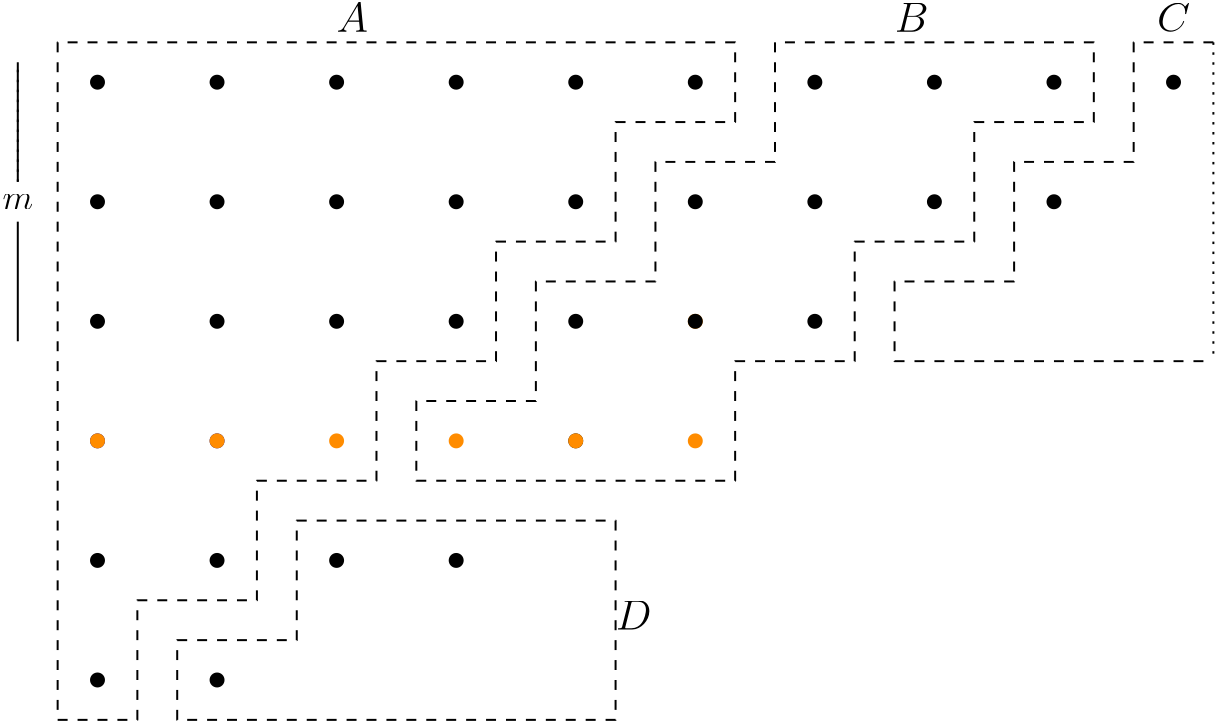}
			\caption{A Rascoe partition}\label{figure3}
		\end{center}
	\end{figure}
	Region $A$ is formed by taking $k$ nodes from the $k^{\textup{th}}$ part of the partition (counted from below), for $1\leq k\leq n$, which is guaranteed by the fact that a Rascoe partition has distinct parts. Note that the sub-partition forming the region $B$ has $m+1$ parts, each equal to $m$. This is because, region $B$ is formed out of those parts in the partition which are $\geq n$ and distinct. The sub-partition in region $C$ is generated by $1/(q)_m$ as there can be at most $m$ number of parts. Finally, Region $D$ has at most $(n-m-1)$ parts, each $\le m$, and is, therefore, generated by $\qbinom{n-1}{m}$ \cite[p.~33, Equation (3.2.1)]{gea}. Thus, a Rascoe partition with $n$ number of parts is generated by
	\begin{align*}
		q^{n(n+1)/2} \sum_{m=0}^{n-1} \qbinom{n-1}{m} \frac{q^{m(m+1)}}{(q)_m}.
	\end{align*}
	Therefore,
		\begin{align*}
		\sum_{N=1}^{\infty} a(N) q^N &= \sum_{n=1}^{\infty} q^{n(n+1)/2} \sum_{m=0}^{n-1} \qbinom{n-1}{m} \frac{q^{m(m+1)}}{(q)_m}. 
		\end{align*}
	Replacing $m$ by $m-1$ in the above equation establishes \eqref{theoremrp}.
	
	To prove (ii), consider a typical non-Rascoe partition shown in Figure \ref{figure 4}, where the number of parts,  that is, $n=6$ in this case, is not a part. Clearly, $1\leq m\leq n$.
		\begin{figure}[hbt!]
		\begin{center}
			\includegraphics[width=0.50\textwidth]{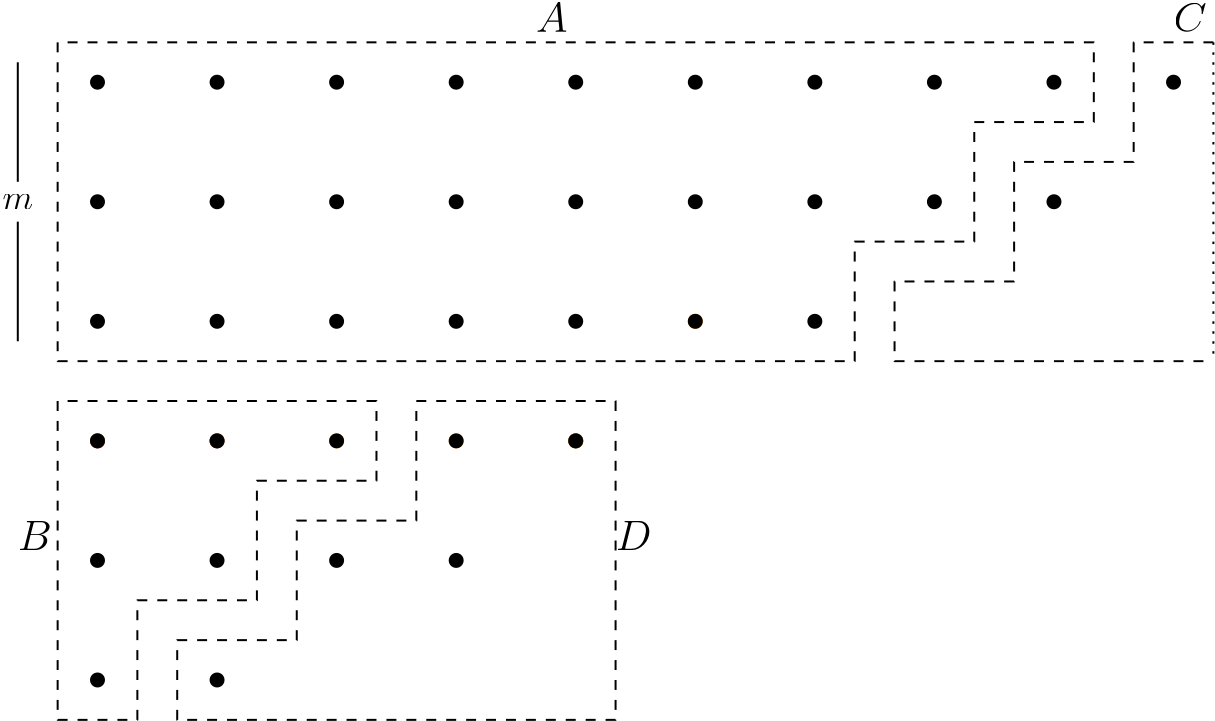}
			\caption{A non-Rascoe partition}\label{figure 4}
		\end{center}
	\end{figure}
	We again divide the Ferrers diagram into four different regions - $A, B, C$ and $D$. Since the region $A$ is formed from parts $>n$ which are also distinct, it contains exactly the parts $n+1, \cdots, n+m$. Therefore, it contributes $q^{nm+m(m+1)/2}$ towards the generating function. Since the region $B$ is formed by taking $k$ nodes from $k^{\textup{th}}$ part of the partition, where $1\leq k\leq n-m$, it will contribute $q^{(n-m)(n-m+1)/2}$. Region $C$ is generated by $1/(q)_m$ as there can be at most $m$ number of parts. Finally, region $D$ is a sub-partition consisting of at most $(n-m)$ parts, each $\le (m-1)$, and is thus generated by $\qbinom{n-1}{m-1}$.
	 
	 Therefore, the non-Rascoe partitions with $n$ number of parts are generated by
	 \begin{align*}
	 	\sum_{m=1}^{n} \qbinom{n-1}{m-1} \frac{q^{nm+m(m+1)/2}q^{(n-m)(n-m+1)/2}}{(q)_m}.
	 \end{align*}
Hence,
	 \begin{align*}
	 	\sum_{N=1}^{\infty} b(N) q^N = \sum_{n=1}^{\infty} q^{n(n+1)/2} \sum_{m=1}^{n} \qbinom{n-1}{m-1} \frac{q^{m^2}}{(q)_m}.
	 \end{align*}
\end{proof}
Next, we make use of the principle of conjugation to obtain an identity which allows us to derive a new representation of $\sigma_2(q)$. 
\begin{theorem}\label{theorem0}
	\begin{align}\label{theoremsmall1}
		\sum_{n=0}^{\infty} \frac{z^nq^{n^2}}{(-q)_n} = \frac{1}{(-q)_{\infty}} \sum_{n=0}^{\infty} q^{n(n+1)/2} \sum_{i=0}^{n} \frac{z^iq^{i(i-1)/2}}{(q)_{n-i}}.
	\end{align}
	\begin{align}\label{theoremsmall2}
		\sum_{n=0}^{\infty} \frac{z^nq^{n^2}}{(-q)_n} = \frac{(-z^{-1}q)_{\infty}}{(-q)_{\infty}} \sum_{n=0}^{\infty} (-z)_nq^{n(n+1)/2} - \frac{1}{(-q)_{\infty}} \sum_{n=0}^{\infty} q^{n(n+1)/2} \sum_{i=1}^{\infty} \frac{z^{-i}q^{i(i+1)/2}}{(q)_{n+i}}.
	\end{align}
\end{theorem}
	\begin{proof}
	Let $P(i,N)$ denote the number of partitions of $N$ into distinct parts, except the smallest part, say $i$, which repeats exactly $i$ number of times. A partition enumerated by $P(i,N)$ is given in Figure \ref{figure 5}.
	\begin{figure}[hbt!]
		\begin{center}
			\includegraphics[width=0.30\textwidth]{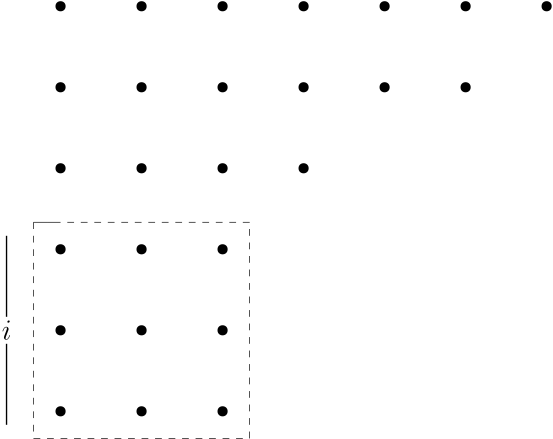}
			\caption{A partition enumerated by $P(i,N)$}\label{figure 5}
		\end{center}
	\end{figure}

	Then $q^{i^2}$ generates the dotted region and $(-q^{i+1})_{\infty}$ will generate the portion above it.
With $z$ keeping track of the smallest part, the generating function of $P(i,N)$ thus turns out to be
	\begin{align}\label{GenP}
		\sum_{N=0}^{\infty} \sum_{i=0}^{N} P(i,N) z^i q^N = \sum_{i=0}^{\infty} z^i q^{i^2} (-q^{i+1})_{\infty}.
	\end{align}
	We now take the conjugate of the above partition; see Figure \ref{figure 6}.
	\begin{figure}[hbt!]
		\begin{center}
			\includegraphics[width=0.25\textwidth]{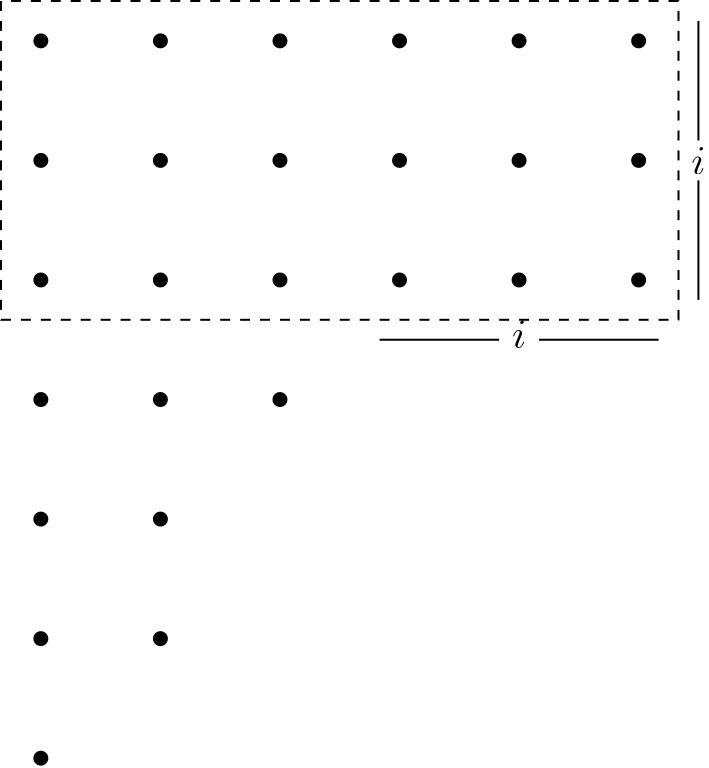}
			\caption{Conjugate of the partition in Figure \ref{figure 5}}\label{figure 6}
		\end{center}
	\end{figure}
In this new partition, the largest part, say $n$, repeats $i$ number of times. No number $>(n-i)$ and $<n$ can be a part since the smallest part $i$ in the original partition appears exactly $i$ number of times. The portion below the dotted region must contain all parts $\le (n-i)$. 
	
	Thus, the dotted region contributes $q^{ni}$ towards the generating function, and $\frac{q^{(n-i)(n-i+1)/2}}{(q)_{n-i}}$  generates the portion below it. Therefore, the corresponding generating function, where $z$ keeps track of the number of times the largest part repeats, is given by
	\begin{align*}
		\sum_{n=0}^{\infty} \sum_{i=0}^{n} \frac{z^iq^{ni}q^{(n-i)(n-i+1)/2}}{(q)_{n-i}} = \sum_{n=0}^{\infty} q^{n(n+1)/2} \sum_{i=0}^{n} \frac{z^iq^{i(i-1)/2}}{(q)_{n-i}}.
	\end{align*}
	Using the principle of conjugation, \eqref{GenP} and the above equation gives (\ref{theoremsmall1}).
	
	To derive (\ref{theoremsmall2}), replace $i$ by $n-i$ on right-hand side of (\ref{theoremsmall1}) and then use Euler's theorem \cite[p.~9, Corollary (1.3.2)]{ntsr}
	\begin{equation*} \sum_{k=0}^{\infty}\frac{(-w)^kq^{k(k-1)/2}}{(q)_k}=(w)_{\infty}\hspace{8mm}(|w|<\infty)
	\end{equation*}
	to get
	\begin{align*}
		\sum_{n=0}^{\infty} \frac{z^nq^{n^2}}{(-q)_n}&= \frac{1}{(-q)_{\infty}} \sum_{n=0}^{\infty} q^{n(n+1)/2} \sum_{i=0}^{n} \frac{z^{n-i}q^{(n-i)(n-i-1)/2}}{(q)_{i}}\\ &= \frac{1}{(-q)_{\infty}} \sum_{n=0}^{\infty} z^nq^{n^2} \sum_{i=0}^{n} \frac{(z^{-1}q^{1-n})^iq^{i(i-1)/2}}{(q)_{i}}\\ &= \frac{1}{(-q)_{\infty}} \sum_{n=0}^{\infty} z^nq^{n^2} \bigg\{ (-z^{-1}q^{1-n})_{\infty}-\sum_{i=n+1}^{\infty} \frac{(z^{-1}q^{1-n})^iq^{i(i-1)/2}}{(q)_{i}} \bigg\}\\ &= \frac{(-z^{-1}q)_{\infty}}{(-q)_{\infty}} \sum_{n=0}^{\infty} z^nq^{n^2} (-z^{-1}q^{1-n})_n - \frac{1}{(-q)_{\infty}} \sum_{n=0}^{\infty} z^nq^{n^2} \sum_{i=n+1}^{\infty} \frac{(z^{-1}q^{1-n})^iq^{i(i-1)/2}}{(q)_{i}}.
	\end{align*}
	Now use the fact $(-z^{-1}q^{1-n})_n=(z^{-1}q)^nq^{-n(n+1)/2}(-z)_n$ in the first sum and replace $i$ by $i+n$ in the second sum, and then simplify to obtain (\ref{theoremsmall2}).
\end{proof}
Two new representations of $\sigma_2(q)$ are now obtained.
	\begin{corollary}
	\begin{align}\label{theoremsigma3}
		\sigma_2(q) = \frac{1}{(-q)_{\infty}} \sum_{n=0}^{\infty} q^{n(n+1)/2} \sum_{i=0}^{n} \frac{(-1)^iq^{i(i-1)/2}}{(q)_{n-i}}.
	\end{align}
	\begin{align*}
		\sigma_2(q) = \frac{(q)_{\infty}}{(-q)_{\infty}} - \frac{1}{(-q)_{\infty}} \sum_{n=0}^{\infty} q^{n(n+1)/2} \sum_{i=1}^{\infty} \frac{(-1)^iq^{i(i+1)/2}}{(q)_{n+i}}.
	\end{align*}
\end{corollary}
\begin{proof}
	Let $z=-1$ in (\ref{theoremsmall1}) and (\ref{theoremsmall2}).
\end{proof}
As one of the final steps towards proving Theorem \ref{theorem main}, we derive a finite analogue of an equivalent form of an identity of Bhargava and Adiga. This equivalent form is due to Srivastava \cite[Equation (6)]{srivastava}.
	\begin{theorem}\label{finite srivastava}
	\begin{align}\label{finite4}
		\sum_{m=0}^{n} \qbinom{n}{m} \frac{(-\lambda/a)_ma^mq^{m(m+1)/2}}{(-bq)_m} = \frac{(-aq)_n}{(-bq)_n} \sum_{m=0}^{n} \qbinom{n}{m} \frac{(-\lambda/b)_mb^mq^{m(m+1)/2}}{(-aq)_m}.
	\end{align}
\end{theorem}
\begin{proof}
	Andrews' finite Heine transformation \cite[Corollary 3, Equation (2.7)]{AndFinite} is given by
	\begin{align*}
		\sum_{m=0}^{\infty} \frac{(q^{-n})_m(\alpha)_m(\beta)_m}{(q)_m(\gamma)_m(q^{1-n}/\tau)_m} q^m = \frac{(\gamma/\beta)_n(\beta\tau)_n}{(\gamma)_n(\tau)_n} \sum_{m=0}^{\infty} \frac{(q^{-n})_m(\alpha\beta\tau/\gamma)_m(\beta)_m}{(q)_m(\beta\tau)_m(\beta q^{1-n}/\gamma)_m} q^m.
	\end{align*}
	Now replace $\tau$ by $\tau/\beta$, let $\beta\to\infty$ and then use the facts
	\begin{align*}
		\lim_{\beta\to\infty} \frac{(\beta)_m}{(\beta q^{1-n}/\tau)_m} = {\tau}^mq^{mn-m},\hspace{5mm}	(q^{-n})_m = (-1)^mq^{-mn+m(m-1)/2} \frac{(q)_n}{(q)_{n-m}},
	\end{align*}
	to obtain
	\begin{align*}
		\sum_{m=0}^{n} \qbinom{n}{m} \frac{(\alpha)_m}{(\gamma)_m} (-1)^m{\tau}^mq^{m(m-1)/2} = \frac{(\tau)_n}{(\gamma)_n} \sum_{m=0}^{n} \qbinom{n}{m} \frac{(\alpha\tau/\gamma)_m}{(\tau)_m} (-1)^m{\gamma}^mq^{m(m-1)/2}.
	\end{align*}
	Now, let $\alpha=-\lambda/a,\gamma=-bq$ and $\tau=-aq$ so as to arrive at \eqref{finite4}.
\end{proof}
\begin{remark}
	As indicated before the statement of Theorem \ref{finite srivastava}, letting $a=-x/q, b=-y/q$, replacing $\lambda$ by $\lambda/q$, and then letting $n\to\infty$ yields Srivastava's equivalent form of a finite identity due to Bhargava and Adiga.
\end{remark}
\begin{corollary}
	\begin{align}\label{finite41}
		\sum_{m=0}^{n} \qbinom{n}{m} \frac{{\lambda}^mq^{m^2}}{(-bq)_m} = \frac{1}{(-bq)_n} \sum_{m=0}^{n} \qbinom{n}{m} (-\lambda/b)_mb^mq^{m(m+1)/2}.
	\end{align}
\end{corollary}
\begin{proof}
	Let $a\to 0$ in (\ref{finite4}) to get \eqref{finite41}. 
\end{proof}

	\begin{theorem}\label{theoremsigma5thm}
	\begin{align*}
		\sigma_2(q) = \frac{1}{(-q)_{\infty}} \sum_{n=0}^{\infty} q^{n(n+1)/2} \sum_{m=0}^{n} \qbinom{n-1}{m-1} \frac{q^{m^2}}{(q)_m}.
	\end{align*}
\end{theorem}
	\begin{proof}
		From \eqref{finite41},
			\begin{align*}
			\sum_{m=0}^{n} \qbinom{n}{m} \left(-bq^{m+1}\right)_{n-m}{\lambda}^mq^{m^2} =  \sum_{m=0}^{n} \qbinom{n}{m} (-\lambda/b)_mb^mq^{m(m+1)/2}.
		\end{align*}
	Now, put $\lambda=1$ and $b=-q^{-1}$ in the above equation to get
	\begin{align*}
		\sum_{m=0}^{n} \qbinom{n}{m} q^{m^2}\left(q^{m}\right)_{n-m} &= (q)_n\sum_{m=0}^{n} \frac{(-1)^mq^{m(m-1)/2}}{(q)_{n-m}}.
	\end{align*}
	We now apply this identity to rephrase (\ref{theoremsigma3}) in the form
	\begin{align*}
	\sigma_2(q) &= \frac{1}{(-q)_{\infty}} \sum_{n=0}^{\infty} \frac{q^{n(n+1)/2}}{(q)_n} \sum_{m=0}^{n} \qbinom{n}{m} q^{m^2}\left(q^{m}\right)_{n-m}\\  &= \frac{1}{(-q)_{\infty}} \sum_{n=0}^{\infty} q^{n(n+1)/2} \sum_{m=0}^{n} \qbinom{n-1}{m-1} \frac{q^{m^2}}{(q)_m}.
	\end{align*}
	This proves the result.
\end{proof}
Armed with the results developed so far, we are now ready to prove our main theorem. 
\begin{proof}[Theorem \textup{\ref{theorem main}}][]
The result follows from \eqref{theoremnrp} and Theorem \ref{theoremsigma5thm}.
\end{proof}

\begin{proof}[Theorem \textup{\ref{non-rascoe odd}}][]
	From Theorem \ref{theorem main}, 
	\begin{align}\label{fir}
		\sum_{n=0}^{\infty} b(n) q^n = (-q)_{\infty} \sum_{n=0}^{\infty} \frac{(-1)^nq^{n^2}}{(-q)_n} \equiv (q)_{\infty}\sum_{n=0}^{\infty} \frac{q^{n^2}}{(q)_n} \pmod{2}.
	\end{align}
	Now using the first Rogers--Ramanujan identity \eqref{rr1} and the following special case of Jacobi's triple product identity \cite[p.~21, Theorem 2.8]{gea}, namely,
	\begin{align}\label{jtpi1}
		\sum_{n=-\infty}^{\infty}(-1)^nq^{n(5n+1)/2}=(q^2;q^5)_{\infty}(q^3;q^5)_{\infty}(q^5;q^5)_{\infty},
	\end{align}
	we arrive at
	\begin{align}\label{sec}
		(q)_{\infty}\sum_{n=0}^{\infty} \frac{q^{n^2}}{(q)_n} = \frac{(q)_{\infty}}{(q;q^5)_{\infty}(q^4;q^5)_{\infty}} = (q^2;q^5)_{\infty}(q^3;q^5)_{\infty}(q^5;q^5)_{\infty} = \sum_{n=-\infty}^{\infty} (-1)^nq^{n(5n+1)/2}.
	\end{align}
	From \eqref{fir} and \eqref{sec}, we conclude that $b(n)$ is odd iff  $n=\frac{m(5m+1)}{2}, m\in\mathbb{Z}$.
\end{proof}

\begin{proof}[Theorem \textup{\ref{p(i,n)}}][]
	Let $z=-1$ in (\ref{GenP}) and make use of Theorem \ref{theorem main} to obtain the desired result.
\end{proof}

\section{Hecke-Rogers type representations of $\sigma_2(q)$}\label{analytic}

The two representations of $\sigma_2(q)$ of Hecke-Rogers type as well as an identity between two finite analogues of Fine's function form the content of this section.
\begin{proof}[Theorem \textup{\ref{hecke-type}}][]
We first prove \textup{(i)}. Use \cite[Theorem 1.9]{Liu} with $\alpha,c=-1$ and $a,b\to0$ to get
\begin{align*}
	\sum_{n=0}^{\infty} \frac{(-1)^nq^{n^2}}{(-q)_n}=\frac{1}{(-q)_{\infty}} \bigg\{ 1+ \sum_{n=1}^{\infty} \sum_{j=0}^{n} \frac{(-1)_j(1+q^{2n})(-1)^{n+j}q^{(n(5n-1)/2)+j-jn}}{(q)_j(1+q^n)} \bigg\}.
\end{align*}
To prove $\textup{(ii)}$, let $y=-1$ in \cite[Theorem 18, Equation (12.3)]{AndPPI} to get
\begin{align}\label{appliedthm18}
	\sum_{n=0}^{\infty} \frac{a^nq^{n^2}}{(-q)_n}=\frac{1}{(aq)_{\infty}} \bigg\{ 1+\sum_{n=1}^{\infty} \frac{(-a)^nq^{2n^2}(a^2q^2;q^2)_{n-1}(1-aq^{2n})(1+a)}{(q^2;q^2)_n} \sum_{m=0}^{\infty} \frac{(q^{-n})_m(aq^n)_m(-q)^m}{(-a)_m(q)_m} \bigg\}.
\end{align}
Invoke Heine's second iterate \cite[p.~38]{gea}, that is,
\begin{align}\label{second iterate}
	\sum_{n=0}^{\infty}\frac{(\alpha)_n(\beta)_nt^n}{(\gamma)_n(q)_n}=\frac{(\gamma/\beta)_{\infty}(\beta t)_\infty}{(\gamma)_\infty(t)_\infty}\sum_{n=0}^{\infty}\frac{(\beta)_n(\alpha\beta t/\gamma)_n(\gamma/\beta)^n}{(q)_n(\beta t)_n},
\end{align}
with $\alpha=aq^n$, $\beta=q^{-n}, \gamma=-a$ and $t=-q$ so as to get
\begin{align}\label{abc1}
	\lim_{a\to-1} (1+a)\sum_{m=0}^{\infty} \frac{(q^{-n})_m(aq^n)_m(-q)^m}{(-a)_m(q)_m} &= \lim_{a\to-1} (1+a) \frac{(-aq^n)_{\infty}(-q^{1-n})_{\infty}}{(-a)_{\infty}(-q)_{\infty}} \sum_{m=0}^{\infty} \frac{(q^{-n})_m(-aq^n)^m}{(-q^{1-n})_m}\nonumber\\
	&= \frac{(-q^{1-n})_{n}}{(q)_{n-1}} \sum_{m=0}^{\infty} \frac{(q^{-n})_m}{(-q^{1-n})_m} q^{nm}.
\end{align}
Now use Andrews' ``incomplete'' summation \cite[Equation (2.15)]{AndPTF} with $a=q^{-n},b=-q^{1-n},N=n-1$ to get
\begin{align*}
	\frac{(-q^{1-n})_{n}}{(q)_{n-1}} \sum_{m=0}^{\infty} \frac{(q^{-n})_m}{(-q^{1-n})_m} q^{nm} &= \frac{(-q^{1-n})_{n}(1+q^n)(-1)^{n-1}q^{n^2-n}}{(-1)_n} \sum_{m=0}^{n-1} \frac{(-1)_m}{(q)_m} (-q^{-n})^m \\
	&= (1+q^n)(-1)^{n-1}q^{n(n-1)/2} \sum_{m=0}^{n-1} \frac{(-1)_m}{(q)_m} (-q^{-n})^m.
\end{align*}
Substituting the above equation in (\ref{abc1}) gives
\begin{align*}
	\lim_{a\to-1} (1+a)\sum_{m=0}^{\infty} \frac{(q^{-n})_m(aq^n)_m(-q)^m}{(-a)_m(q)_m} = (1+q^n)(-1)^{n-1}q^{n(n-1)/2} \sum_{m=0}^{n-1} \frac{(-1)_m}{(q)_m} (-q^{-n})^m.
\end{align*}
Now let $a\to-1$ in (\ref{appliedthm18}) and use the above limit evaluation to get
\begin{align*}
	\sum_{n=0}^{\infty} \frac{(-1)^nq^{n^2}}{(-q)_n}=\frac{1}{(-q)_{\infty}} \bigg\{ 1-\sum_{n=1}^{\infty} \frac{(-1)^nq^{n(5n-1)/2}(1+q^{2n})}{(1-q^n)} \sum_{m=0}^{n-1} \frac{(-1)_m}{(q)_m} (-q^{-n})^m \bigg\}.
\end{align*}
\end{proof}

We now prove \eqref{theoremfinite1} with the help of which one can derive (i) of  Theorem \ref{hecke-type} from (ii) or vice-versa.
 
\begin{proof}[ \textup{\ref{theoremfinite1}}][]
	Use \cite[Equation (4.6)]{And5n7} with $a,b=-1$ to get
	\begin{align}\label{finite1}
		1+\sum_{j=1}^{n-1} \frac{(-q)_{j-1}(1+q^{2j})(-1)_jq^{-j^2}}{(q)_j^2}=\frac{(-1)^{n-1}(-q)_{n-1}}{(q)_{n-1}} \sum_{j=0}^{n-1} \frac{(-1)_j(-1)^jq^{j-jn}}{(q)_j}.
	\end{align}
	Now add  $\frac{(-q)_{n-1}(1+q^{2n})(-1)_nq^{-n^2}}{(q)_n^2}$ on both sides, and add and subtract $\frac{(-q)_{n-1}(-1)_nq^{n-n^2}}{(q)_{n-1}(q)_n}$ on the right-hand side so that
	\begin{align}\label{finite2}
		&1+\sum_{j=1}^{n} \frac{(-q)_{j-1}(1+q^{2j})(-1)_jq^{-j^2}}{(q)_j^2}\nonumber\\  &=\frac{(-q)_{n-1}(-1)_nq^{-n^2}}{(q)_n}\bigg\{ \frac{(1+q^{2n})}{(q)_n} + \frac{q^n}{(q)_{n-1}} \bigg\}+\frac{(-1)^{n-1}(-q)_{n-1}}{(q)_{n-1}} \sum_{j=0}^{n} \frac{(-1)_j(-1)^jq^{j-jn}}{(q)_j}\nonumber\\ &=\frac{(-q)_n(-1)_nq^{-n^2}}{(q)_n^2}-\frac{(-1)^n(-q)_{n-1}}{(q)_{n-1}} \sum_{j=0}^{n} \frac{(-1)_j(-1)^jq^{j-jn}}{(q)_j}.
	\end{align}
	Now replace $n$ by $n+1$ in (\ref{finite1}) and then separate the $j=n$ term from the right-hand side to get
	\begin{align}\label{finite3}
		1+\sum_{j=1}^{n} \frac{(-q)_{j-1}(1+q^{2j})(-1)_jq^{-j^2}}{(q)_j^2}
		=\frac{(-q)_n(-1)_nq^{-n^2}}{(q)_n^2} + \frac{(-1)^n(-q)_n}{(q)_n} \sum_{j=0}^{n-1} \frac{(-1)_j(-1)^jq^{-jn}}{(q)_j}.
	\end{align}
	Now compare (\ref{finite2}) and (\ref{finite3}) to get the required result.
\end{proof}
\begin{remark}
One can also prove the identity in Theorem \ref{theoremfinite1} by using the recurrence relation for the finite analogue of Fine's function, namely $F(a, b, t, N)$, given in \cite[Lemma 2.1]{andrews-bell}.

\end{remark}
\section{A generalization of $\sigma_2(q)$ and generalized non-Rascoe partitions}\label{gen rascoe}

Let $\ell\in\mathbb{N}\cup\{0\}$. We define a generalized Rogers--Ramanujan partition of a number $N$ as a partition of $N$ into parts $>\ell$ and in which the difference between any two parts is greater than or equal to 2. 

Garrett, Ismail and Stanton \cite[Equation (3.5)]{garrett-ismail-stanton} showed that the generating function of the number of generalized Rogers--Ramanujan partitions, that is $\sum_{n=0}^{\infty}\frac{q^{n^2+\ell n}}{(q)_n}$, satisfies the generalized Rogers--Ramanujan identity
\begin{align}\label{gis identity}
\sum_{n=0}^{\infty}\frac{q^{n^2+\ell n}}{(q)_n}=\frac{(-1)^{\ell}q^{-\ell(\ell-1)/2}c_\ell(q)}{(q;q^5)_\infty(q^4;q^5)_\infty}-\frac{(-1)^{\ell}q^{-\ell(\ell-1)/2}d_\ell(q)}{(q^2;q^5)_\infty(q^3;q^5)_\infty},
\end{align}
where
\begin{align*}
c_\ell(q):=\sum_{j}(-1)^{j}q^{\frac{j(5j-3)}{2}}\qbinom{\ell-1}{\left\lfloor\frac{\ell+1-5j}{2}\right\rfloor},\hspace{5mm}
d_\ell(q):=\sum_{j}(-1)^{j}q^{\frac{j(5j+1)}{2}}\qbinom{\ell-1}{\left\lfloor\frac{\ell-1-5j}{2}\right\rfloor}
\end{align*}
are Laurent polynomials in $q$.

Let $R_{\ell}(m,N)$ denotes the number of generalized Rogers--Ramanujan partitions of $N$ with rank $m$. In what follows, we use simple combinatorics to obtain its generating function.
\begin{theorem}
	For $\ell\in\mathbb{N}\cup\{0\}$,
	\begin{align}\label{theoremgrrp}
		\sum_{N=0}^{\infty} \sum_{m=-\infty}^{\infty} R_{\ell}(m,N) z^m q^N = \sum_{n=0}^{\infty} \frac{z^{n+\ell-1}q^{n^2+\ell n}}{(zq)_n}.
	\end{align}
\end{theorem}
\begin{proof}
	Let $n$ be the number of parts and $k$ be the largest part of a partition enumerated by $R_{\ell}(m,N)$.
	\begin{figure}[hbt!]
		\begin{center}
			\includegraphics[width=0.50\textwidth]{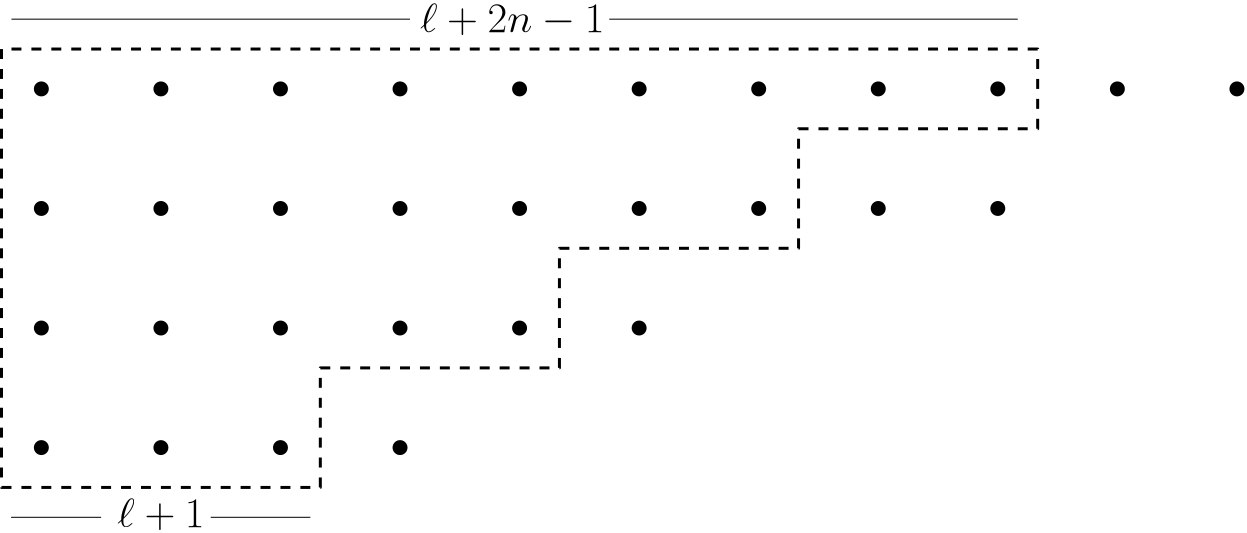}
			\caption{A generalized Rogers--Ramanujan partition}
			\label{fig3}
		\end{center}
	\end{figure}
	
	Observe from Figure \ref{fig3} that the number of nodes in the dotted region is equal to $(\ell+1)+(\ell+3)+\cdots+(\ell+2n-1)=n^2+\ell n$. Also, $1/(zq)_n$ generates the partition to the right of the dotted region with $z$ keeping track of $k-(\ell+2n-1)$.  Thus, $\frac{z^{n+\ell-1}q^{n^2+\ell n}}{(zq)_n}$ gives us the generalized Rogers--Ramanujan partitions into $n$ number of parts, with each part $>\ell$, and where $z$ keeps track of the rank of the partition, which is $k-n$. The desired generating function is then obtained by summing these terms corresponding to $n$ from $0$ to $\infty$.
\end{proof}

	\begin{theorem}
	Let $L_{\ell}(m,N)$ be the number of partitions of a number $N$ into $m$ parts, with largest part, say $n$, repeating at least $n+\ell$ times. Then, $L_{\ell}(m,N)$ is equal to the generalized Rogers--Ramanujan partitions of the number $N$ with rank $m-1$. In other words, $L_{\ell}(m,N)=R_{\ell}(m-1,N)$.
\end{theorem}
\begin{proof}
	This result is proved by showing that the generating function of $L_{\ell}(m,N)$ is
	\begin{align*}
		\sum_{N=0}^{\infty} \sum_{m=0}^{\infty} L_{\ell}(m,N) z^m q^N = \sum_{n=0}^{\infty} \frac{z^{n+\ell}q^{n^2+\ell n}}{(zq)_n},
	\end{align*}
	and then comparing it with \eqref{theoremgrrp}. 
	\begin{figure}[hbt!]
		\begin{center}
			\includegraphics[width=0.25\textwidth, height=0.25\textheight]{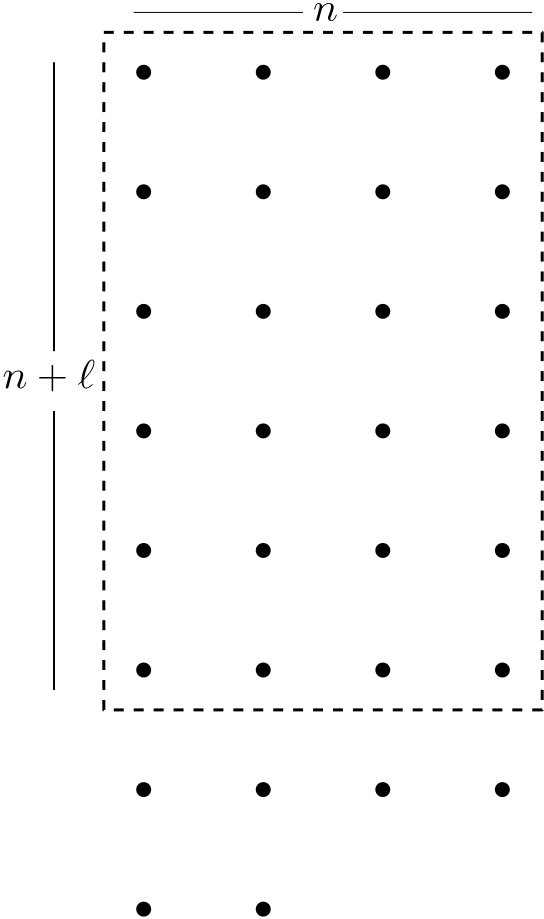}
			\caption{A partition enumerated by $L_{\ell}(m,N)$}
			\label{fig4}
		\end{center}
	\end{figure}
	Indeed, as shown in Figure \ref{fig4}, if $n$ is the largest part of the partition enumerated by $L_\ell(m, N)$, then we construct a rectangle whose width is $n+\ell$. It is clear that $1/(zq)_n$ generates the partition below the rectangle with $z$ keeping track of  the total number of parts minus $(n+\ell)$. So $\frac{z^{n+\ell}q^{n^2+\ell n}}{(zq)_n}$ generates partitions with $n$ as the largest part, and where $z$ keeps track of the number of parts of the partition.

\end{proof}

We now define generalized Rascoe and non-Rascoe partitions. Let $\ell\in\mathbb{N}\cup\{0\}$ and let $n$ denote the number of parts of a partition. A \emph{generalized Rascoe partition} of $N$, associated with $\ell$, is a partition of $N$ into distinct parts containing $n+\ell$ as a part. Similarly a \emph{generalized non-Rascoe partition}  of a number $N$, associated with $\ell$, is a partition of $N$ into distinct parts not containing $n+\ell$ as a part.

Let $a_\ell(N)$ and $b_\ell(N)$ respectively denote the number of generalized Rascoe and non-Rascoe partitions of $N$, associated with $\ell$. Clearly, $a_0(N)=a(N)$ and $b_0(N)=b(N)$, where $a(N)$ and $b(N)$ are defined in the introduction. The generating functions of $a_\ell(N)$ and $b_\ell(N)$ are obtained below using combinatorics.
\begin{theorem}\label{after defn grgnr}
For	$\ell\in\mathbb{N}\cup\{0\}$,
	\begin{align}\label{theoremgrp}
		\sum_{N=1}^{\infty} a_\ell(N) q^N = \sum_{n=1}^{\infty} q^{n(n+1)/2} \sum_{m=0}^{n-1} \qbinom{n+\ell-1}{m+\ell} \frac{q^{(m+\ell)(m+1)}}{(q)_m},
	\end{align}
	\begin{align}\label{theoremgnrp}
		\sum_{N=1}^{\infty} b_\ell(N) q^N = \sum_{n=1}^{\infty} q^{n(n+1)/2} \sum_{m=0}^{n} \qbinom{n+\ell-1}{m+\ell-1} \frac{q^{m^2+\ell m}}{(q)_m}.
	\end{align}
\end{theorem}
\begin{proof}
	Assume $n$ to be the number of parts of a partition of $N$, and $m$ to be the number of parts greater than $n+\ell$.
	
	Consider a generalized Rascoe partition given in Figure \ref{figure grp}, where $\ell=2$ and $n=5$. Note that $n+\ell=7$ is a part. Also, $m=2$ since there are $2$ parts larger than $7$.
	\begin{figure}[hbt!]
		\begin{center}
			\includegraphics[width=0.50\textwidth]{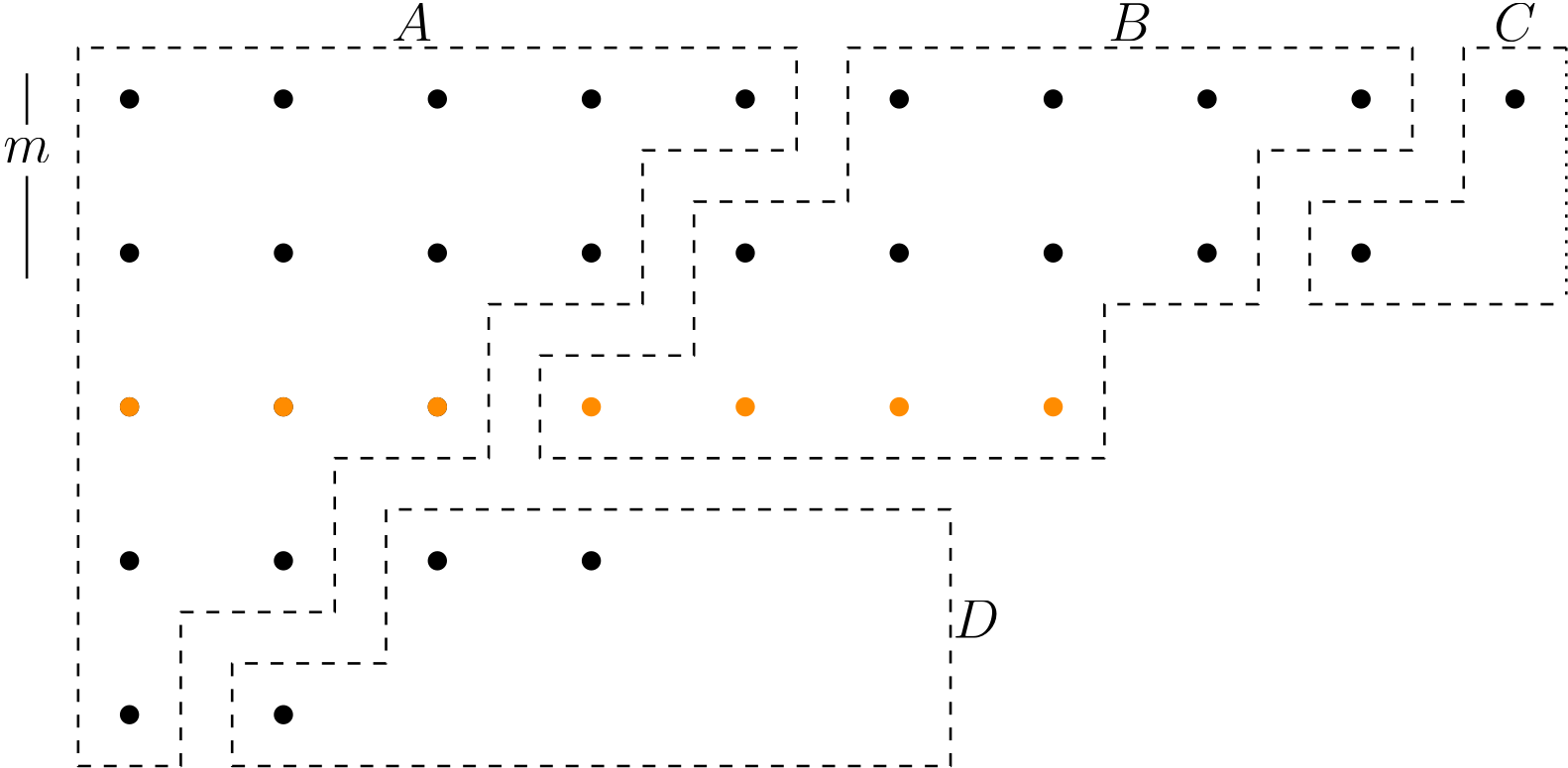}
			\caption{A generalized Rascoe partition}\label{figure grp}
		\end{center}
	\end{figure}
	Note that region $A$ has $n(n+1)/2$ nodes which contribute $q^{n(n+1)/2}$ towards the generating function of $a_\ell(N)$. The region $B$ has $m+1$ rows of exactly $m+\ell$ nodes each, thereby contributing $q^{(m+\ell)(m+1)}$ towards the generating function. Next, the region $C$ contributes $1/(q)_m$ as there can be at most $m$ number of parts. Finally, the region $D$ has at most $n-m-1$ parts, each $\le m+\ell$, and therefore contributes $\qbinom{n+\ell-1}{m+\ell}$. So, the generalized Rascoe partitions of $N$, associated with $\ell$, having $n$ as the number of parts, are generated by
	\begin{align*}
		q^{n(n+1)/2} \sum_{m=0}^{n-1} \qbinom{n+\ell-1}{m+\ell} \frac{q^{(m+\ell)(m+1)}}{(q)_m}.
	\end{align*}
%
This proves \eqref{theoremgrp}. We next prove \eqref{theoremgnrp}.

	Consider a generalized non-Rascoe partition shown in Figure \ref{figure gnrp}. Again, we let $\ell=2$ and $n=5$. Observe that $n+\ell=7$ is not a part of the partition. We have $m=2$ since there are two parts greater than $7$.
	\begin{figure}[hbt!]
		\begin{center}
			\includegraphics[width=0.50\textwidth]{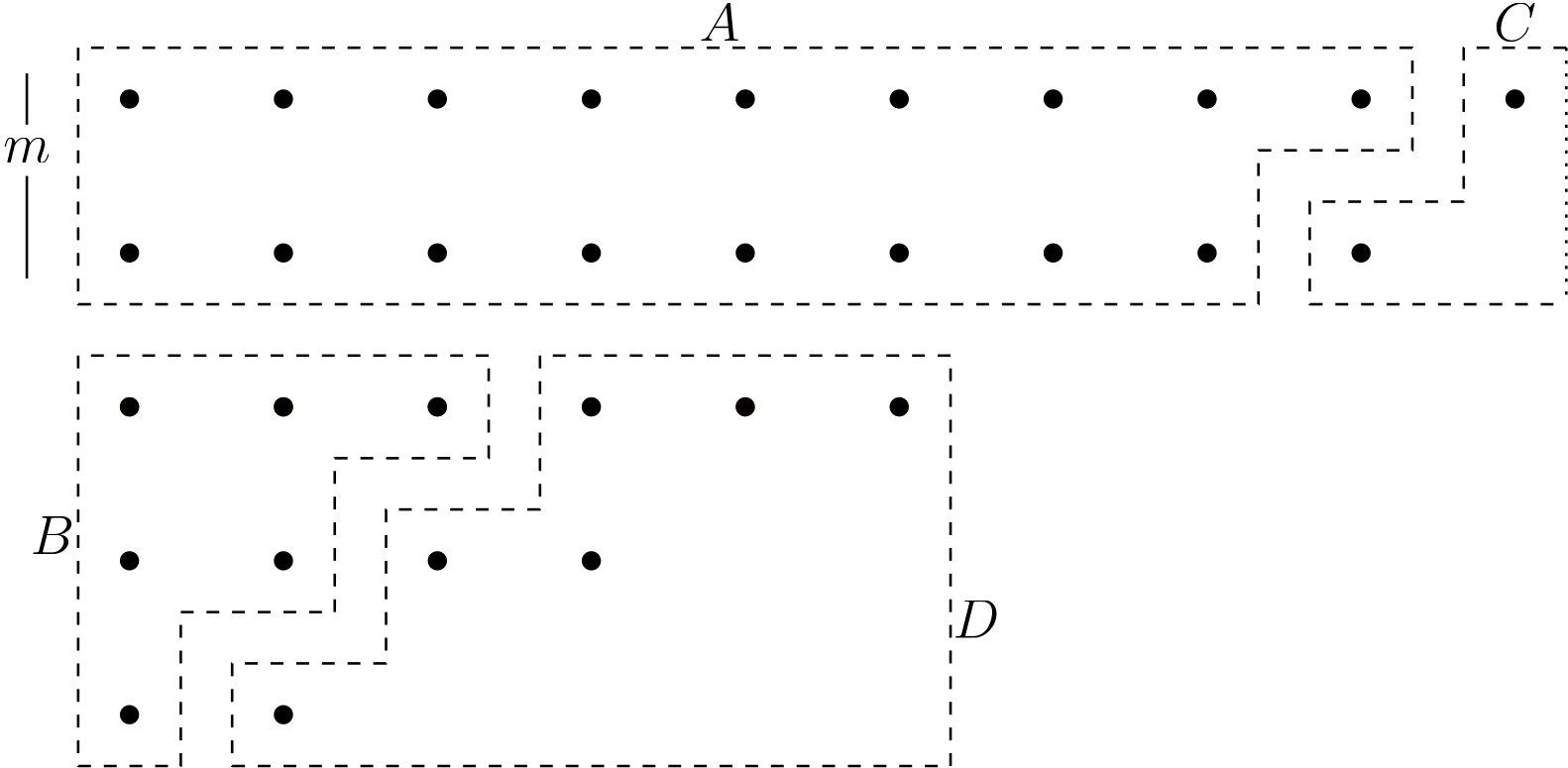}
			\caption{A generalized non-Rascoe partition}\label{figure gnrp}
		\end{center}
	\end{figure}
	Note that the region $A$ consists of exactly $m$ parts, namely, $n+\ell+1,n+\ell+2,\cdots,n+\ell+m$, thereby contributing $q^{nm+\ell m+m(m+1)/2}$ to the generating function of $b_\ell(N)$. The region $B$ consists of $n-m$ rows with exactly $(n-m)(n-m+1)/2$ nodes, contributing to $q^{(n-m)(n-m+1)/2}$. From region $C$, we get $1/(q)_m$ whereas region $D$ has at most $n-m$ parts, each $\le m+\ell-1$,  and hence gives $\qbinom{n+\ell-1}{m+\ell-1}$. Thus, the generalized non-Rascoe partitions, associated with $\ell$, and having $n$ number of parts are generated by
	\begin{align*}
		\sum_{m=0}^{n} \qbinom{n+\ell-1}{m+\ell-1} \frac{q^{nm+\ell m+m(m+1)/2}q^{(n-m)(n-m+1)/2}}{(q)_m}.
	\end{align*}
	Summing over $n$ from $1$ to $\infty$ leads us to \eqref{theoremgnrp}.
\end{proof}

\begin{theorem}\label{added just before submitting}
	\begin{align}\label{theoremsmall1 ell}
	\sum_{n=0}^{\infty} \frac{z^nq^{n^2+\ell n}}{(-q)_n} = \frac{z^{-\ell}}{(-q)_{\infty}} \sum_{n=0}^{\infty} q^{n(n+1)/2-\ell n} \sum_{i=\ell}^{n} \frac{z^iq^{i(i-1)/2}}{(q)_{n-i}}.
\end{align}
\end{theorem}
\begin{proof}
Let $P_{\ell}(j, N)$ be the number of partitions of $N$ into distinct parts with the exception that the smallest part, say $j$, is allowed to repeat exactly $j+\ell$ number of times. Arguing similarly as in the proof of  \eqref{theoremsmall1}, one obtains
	\begin{align}\label{GenPell}
	\sum_{N=0}^{\infty} \sum_{i=0}^{N} P_\ell(i,N) z^i q^N = \sum_{i=0}^{\infty} z^i q^{i^2+\ell i} (-q^{i+1})_{\infty}.
\end{align}
 Indeed, following the remainder of the proof of \eqref{theoremsmall1}, one is also led to
\begin{align*}
\sum_{n=0}^{\infty} \frac{z^nq^{n^2+\ell n}}{(-q)_n} = \frac{q^{\ell(\ell-1)/2}}{(-q)_{\infty}} \sum_{n=0}^{\infty} q^{n(n+1)/2-\ell n} \sum_{i=0}^{n} \frac{z^iq^{i(i-1)/2+\ell i}}{(q)_{n-i-\ell}}.
\end{align*}
Now replace $i$ by $i-\ell$ in the above identity to arrive at \eqref{theoremsmall1 ell}.
\end{proof}
\begin{theorem}
	For $\ell\in\mathbb{Z}^{+}\cup\{0\}$,
	\begin{align}\label{sigmageneral}
		\sum_{n=0}^{\infty} \frac{(-1)^nq^{n^2+\ell n}}{(-q)_n} = \frac{1}{(-q)_{\infty}} \sum_{n=0}^{\infty} q^{n(n+1)/2} \sum_{m=0}^{n} \qbinom{n+\ell-1}{m+\ell-1} \frac{q^{m^2+\ell m}}{(q)_m}.
	\end{align}
\end{theorem}
\begin{proof}
	First let $z=-q^\ell$ in (\ref{theoremsmall1}) so as to get
	\begin{align}\label{xyz}
		\sum_{n=0}^{\infty} \frac{(-1)^nq^{n^2+\ell n}}{(-q)_n} = \frac{1}{(-q)_{\infty}} \sum_{n=0}^{\infty} q^{n(n+1)/2} \sum_{m=0}^{n} \frac{(-1)^mq^{\ell m+m(m-1)/2}}{(q)_{n-m}}.
	\end{align}
	Now, put $\lambda=q^{\ell}$ and $b=-q^{\ell-1}$ in (\ref{finite41}) to obtain
	\begin{align*}
		\sum_{m=0}^{n} \qbinom{n}{m} \frac{q^{m^2+\ell m}}{(q^{\ell})_m} = \frac{(q)_n}{(q^{\ell})_n} \sum_{m=0}^{n} \frac{(-1)^mq^{\ell m+m(m-1)/2}}{(q)_{n-m}}.
	\end{align*}
	Using the above identity to represent the inner sum of (\ref{xyz}) leads to
	\begin{align*}
		\sum_{n=0}^{\infty} \frac{(-1)^nq^{n^2+\ell n}}{(-q)_n} &= \frac{1}{(-q)_{\infty}} \sum_{n=0}^{\infty} \frac{(q^{\ell})_nq^{n(n+1)/2}}{(q)_n} \sum_{m=0}^{n} \qbinom{n}{m} \frac{q^{m^2+\ell m}}{(q^{\ell})_m}\\ &= \frac{1}{(-q)_{\infty}} \sum_{n=0}^{\infty} q^{n(n+1)/2} \sum_{m=0}^{n} \frac{(q^{\ell})_n}{(q^{\ell})_m(q)_{n-m}} \frac{q^{m^2+\ell m}}{(q)_m}\\ &= \frac{1}{(-q)_{\infty}} \sum_{n=0}^{\infty} q^{n(n+1)/2} \sum_{m=0}^{n} \qbinom{n+\ell-1}{m+\ell-1} \frac{q^{m^2+\ell m}}{(q)_m}.
	\end{align*}
\end{proof}

Now we define a rank parity function associated to the generalized Rogers--Ramanujan partitions, namely,
\begin{align*}
	\sigma_{2,\ell}(q):=\sum_{n=0}^{\infty} \frac{(-1)^nq^{n^2+\ell n}}{(-q)_n}.
\end{align*}
Letting $z=-1$ in (\ref{theoremgrrp}), we see that if $\ell$ is even, $\sigma_{2,\ell}(q)$ is the generating function of the excess number of generalized Rogers--Ramanujan partitions corresponding to odd rank over those with even rank, and when $\ell$ is odd, it is the excess number of such partitions with even rank over those with odd rank. Clearly, $\sigma_{2,0}(q)=\sigma_2(q)$.

\begin{theorem}\label{main theorem general ell}
	Let $\ell\in\mathbb{N}\cup\{0\}$. Then,
	\begin{align*}
		\sum_{n=0}^{\infty} b_{\ell}(n) q^n = (-q)_{\infty} \sigma_{2,\ell}(q).
	\end{align*}
	In other words, $(-q)_{\infty} \sigma_{2,\ell}(q)$ generates the generalized non-Rascoe partitions.
\end{theorem}
\begin{proof}
	The result follows from (\ref{theoremgnrp}) and (\ref{sigmageneral}).
\end{proof}
We now obtain a generalization of Theorem \ref{p(i,n)} for any $\ell\in\mathbb{N}\cup\{0\}$.
\begin{theorem}\label{p(i,n) ell}
	Let $P_{\ell}(j,n)$ be defined as in the proof of Theorem \ref{added just before submitting}. Then the excess number of such partitions with even smallest part over those with odd smallest part equals the number of generalized non-Rascoe partitions of $n$, that is,
	\begin{equation*}
		\sum_{j=0}^{n}(-1)^jP_{\ell}(j, n)=b_\ell(n).
	\end{equation*}
\end{theorem}
\begin{proof}
The proof follows from \eqref{GenPell} with $z=-1$ and Theorem \ref{main theorem general ell}. 
\end{proof}

The theorem we now derive allows one to explicitly obtain the parity of $b_\ell(n)$ for any $\ell$. We illustrate this for $\ell=1$ and $2$ in the corollary following its proof.
\begin{theorem}\label{lcong}
Let $\ell\in\mathbb{N}\cup\{0\}$ and let $b_\ell(n)$ denote the generalized non-Rascoe partitions of $n$ defined before Theorem \ref{after defn grgnr}. Also, let $p(N, M, k)$ denote the number of partitions of $k$ into at most $M$ parts, each $\leq N$. Then, we have
\begin{align*}
\sum_{n=0}^{\infty} b_{\ell}(n) q^n
&\equiv\sum_{j,m=-\infty\atop{k=0}}^{\infty}\left(p\left(\ell-1-\left\lfloor\tfrac{\ell+1-5j}{2}\right\rfloor,\left\lfloor\tfrac{\ell+1-5j}{2}\right\rfloor, k\right)+p\left(\ell-1-\left\lfloor\tfrac{\ell-1-5m}{2}\right\rfloor,\left\lfloor\tfrac{\ell-1-5m}{2}\right\rfloor,k\right)\right)\nonumber\\
&\qquad\qquad\times q^{\frac{j(5j-3)}{2}+\frac{m(5m+1)}{2}-\frac{\ell(\ell-1)}{2}+k}\pmod{2}.
\end{align*}
\end{theorem}
\begin{proof}
From \eqref{gis identity} and Theorem \ref{main theorem general ell} and using the two special cases of the Jacobi triple product identity, namely, \eqref{jtpi1} and 
\begin{align*}
	\sum_{n=-\infty}^{\infty}(-1)^nq^{n(5n-3)/2}=(q;q^5)_{\infty}(q^4;q^5)_{\infty}(q^5;q^5)_{\infty},
\end{align*}
we have
\begin{align*}
\sum_{n=0}^{\infty} b_{\ell}(n) q^n &\equiv q^{-\ell(\ell-1)/2}(q)_{\infty}\left\{\frac{(-1)^{\ell}c_\ell(q)}{(q;q^5)_\infty(q^4;q^5)_\infty}+\frac{(-1)^{\ell}d_\ell(q)}{(q^2;q^5)_\infty(q^3;q^5)_\infty}\right\}\pmod{2}\\
&=q^{-\ell(\ell-1)/2}\bigg[\sum_{j}(-1)^{j}q^{\frac{j(5j-3)}{2}}\qbinom{\ell-1}{\left\lfloor\frac{\ell+1-5j}{2}\right\rfloor}\sum_{m=-\infty}^{\infty}(-1)^mq^{m(5m+1)/2}\nonumber\\
&\qquad\qquad\quad+\sum_{j}(-1)^{j}q^{\frac{j(5j+1)}{2}}\qbinom{\ell-1}{\left\lfloor\frac{\ell-1-5j}{2}\right\rfloor}\sum_{m=-\infty}^{\infty}(-1)^mq^{m(5m-3)/2}\bigg]\\
&\equiv\sum_{j,m=-\infty\atop{k=0}}^{\infty}p\left(\ell-1-\left\lfloor\tfrac{\ell+1-5j}{2}\right\rfloor,\left\lfloor\tfrac{\ell+1-5j}{2}\right\rfloor, k\right)q^{\frac{j(5j-3)}{2}+\frac{m(5m+1)}{2}-\frac{\ell(\ell-1)}{2}+k}\nonumber\\
&\quad+\sum_{j,m=-\infty\atop{k=0}}^{\infty}p\left(\ell-1-\left\lfloor\tfrac{\ell-1-5j}{2}\right\rfloor,\left\lfloor\tfrac{\ell-1-5j}{2}\right\rfloor,k\right)q^{\frac{j(5j+1)}{2}+\frac{m(5m-3)}{2}-\frac{\ell(\ell-1)}{2}+k}\pmod{2}
\end{align*}
where in the last step, we used the fact \cite[p.~33, Theorem 3.1]{gea} that
\begin{align*}
	\qbinom{N+M}{M}=\sum_{k=0}^{\infty} p(N,M,k)q^k.
\end{align*}
Now interchange $j$ and $m$ in the second triple sum to arrive at the desired result.
\end{proof}

\begin{corollary}
\textup{(i)} $b_1(n)$  is odd if and only if $n=\frac{m(5m-3)}{2}$, where $m\in\mathbb{Z}$,\\
 \textup{(ii)} $b_2(n)$ is odd if and only if $n=\frac{m(5m+1)}{2}-1$ or $n=\frac{m(5m-3)}{2}-1$, where $m\in\mathbb{Z}$.
\end{corollary}
\begin{proof}
	\textup{(i)} Let $\ell=1$ in Theorem \ref{lcong}, and simplify, thereby obtaining
	\begin{align*}
		\sum_{n=0}^{\infty} b_{1}(n) q^n
		&\equiv\sum_{j,m=-\infty\atop{k=0}}^{\infty}\left(p\left(-\left\lfloor\tfrac{2-5j}{2}\right\rfloor,\left\lfloor\tfrac{2-5j}{2}\right\rfloor, k\right)+p\left(-\left\lfloor\tfrac{-5m}{2}\right\rfloor,\left\lfloor\tfrac{-5m}{2}\right\rfloor,k\right)\right)q^{\frac{j(5j-3)}{2}+\frac{m(5m+1)}{2}+k}\pmod{2}\nonumber\\
		&\equiv\sum_{j=-\infty}^{\infty} q^{\frac{j(5j-3)}{2}}\pmod{2},
	\end{align*}
	where we used the facts $p\left(-\left\lfloor\tfrac{2-5j}{2}\right\rfloor,\left\lfloor\tfrac{2-5j}{2}\right\rfloor, k\right)=0$ for all $j\in\mathbb{Z}$ and  $p\left(-\left\lfloor\tfrac{-5m}{2}\right\rfloor,\left\lfloor\tfrac{-5m}{2}\right\rfloor,k\right)=1$ only when $m=k=0$, and is $0$ for all other values of $m$ and $k$.\\
	
	\textup{(ii)} Put $\ell=2$ in Theorem \ref{lcong}, and simplify to get
	\begin{align*}
		\sum_{n=0}^{\infty} b_{2}(n) q^n
		&\equiv\sum_{j,m=-\infty\atop{k=0}}^{\infty}\left(p\left(1-\left\lfloor\tfrac{3-5j}{2}\right\rfloor,\left\lfloor\tfrac{3-5j}{2}\right\rfloor, k\right)+p\left(1-\left\lfloor\tfrac{1-5m}{2}\right\rfloor,\left\lfloor\tfrac{1-5m}{2}\right\rfloor,k\right)\right)\nonumber\\
		&\qquad\qquad\times q^{\frac{j(5j-3)}{2}+\frac{m(5m+1)}{2}-1+k}\pmod{2}.\nonumber\\
		&\equiv\sum_{m=-\infty}^{\infty} q^{\frac{m(5m+1)}{2}-1}+\sum_{j=-\infty}^{\infty} q^{\frac{j(5j-3)}{2}-1}\pmod{2},
	\end{align*}
	since\footnote{Observe that $p(N,0,0)=p(0,M,0)=1$ for $N, M>0$, which is not clear in \cite[p.~34, Equation (3.2.4)]{gea}.} $p\left(1-\left\lfloor\tfrac{3-5j}{2}\right\rfloor,\left\lfloor\tfrac{3-5j}{2}\right\rfloor, k\right)=1$ only when $j=k=0$, and is $0$ for all other values of $j$ and $k$, and, similarly, $p\left(1-\left\lfloor\tfrac{1-5m}{2}\right\rfloor,\left\lfloor\tfrac{1-5m}{2}\right\rfloor,k\right)=1$ only when $m=k=0$, and is $0$ for all other values of $m$ and $k$.
\end{proof}

\begin{theorem}\label{lm1}
	\begin{align}\label{theoremgsigma1}
		\sigma_{2,\ell}(q)=\frac{1}{(-q)_{\infty}} \bigg\{(-q)_{\ell}-\sum_{n=1}^{\infty} (-1)^{n+\ell}q^{\frac{n(5n-1)}{2}+2\ell n}(1+q^{\ell+2n})(q^{n+1})_{\ell-1} \sum_{m=0}^{n+\ell-1} \frac{(-1)_m}{(q)_m} (-q^{-n})^m \bigg\}.
	\end{align}
\end{theorem}
\begin{proof}
The idea is to take the limit $a\to-q^{\ell}$ on both sides of \eqref{appliedthm18}. Using second iterate of Heine's transformation, that is, \eqref{second iterate}, we get
	\begin{align*}
		\lim_{a\to-q^\ell} (1+a)\sum_{m=0}^{\infty} \frac{(q^{-n})_m(aq^n)_m(-q)^m}{(-a)_m(q)_m} 
		&= \frac{(-q^{1-n})_{n}}{(q^{\ell+1})_{n-1}} \sum_{m=0}^{\infty} \frac{(q^{-n})_m}{(-q^{1-n})_m} q^{nm+\ell m}\nonumber\\
		&= \frac{(q)_{\ell}(1+q^n)(-1)^{n+\ell-1}q^{\ell n+(n(n-1)/2)}}{(-q^n)_{\ell}} \sum_{m=0}^{n+\ell-1} \frac{(-1)_m}{(q)_m} (-q^{-n})^m,
	\end{align*}
	where, in the last step, we used \cite[Equation (2.15)]{AndPTF} with $a=q^{-n},b=-q^{1-n},N=n+\ell-1$. 
	Now, let $a\to-q^\ell$ in (\ref{appliedthm18}) and use the above limit to arrive at \eqref{theoremgsigma1}.
\end{proof}
\begin{theorem}\label{lm2}
	\begin{align*}
		\sigma_{2,\ell}(q)=\frac{1}{(-q)_{\infty}} \bigg\{(-q)_{\ell}+\sum_{n=1}^{\infty} (-1)^nq^{\frac{n(5n-1)}{2}+2\ell n}(1+q^{\ell+2n})(-q^{n+1})_{\ell-1} \sum_{m=0}^{n} \frac{(-1)_m}{(q)_m} (-q^{1-n-\ell})^m \bigg\}.
	\end{align*}
\end{theorem}
\begin{proof}
	Use \cite[Theorem 1.9]{Liu} with $\alpha=-q^{\ell},c=-1$ and $a,b\to0$ to get the desired result.
\end{proof}
We can also derive Theorem \ref{lm2} from Theorem \ref{lm1} by using an identity between two finite sums given in Theorem \ref{theoremgfinite1}. This identity is proved next. We begin with a lemma first. 

\begin{lemma}\label{lfinitelemma}
	For $n,\ell\in\mathbb{N}$,
	\begin{align}\label{lfinitelemma eqn}
		\sum_{j=n+1}^{n+\ell} \frac{(-q)_{j-1}(-1)_j(1+q^{2j})q^{-j^2}}{(q)_j^2} = &\frac{(-1)^{n+\ell}(-q)_{n+\ell}}{(q)_{n+\ell}} \sum_{j=n+1}^{n+\ell} \frac{(-1)_j(-1)^jq^{-j(n+\ell)}}{(q)_j}\nonumber\\ &\quad+\frac{(-1)^n(-q)_n}{(q)_n} \sum_{j=n+1}^{n+\ell} \frac{(-1)_j(-1)^jq^{-jn}}{(q)_j}.
	\end{align}
\end{lemma}
\begin{proof}
	Let $A(q;\ell)$ denote the expression on the left-hand side of \eqref{lfinitelemma eqn}, and let $B(q;\ell)$  and $C(q;\ell)$ denote the first and the second expressions on the right-hand side respectively. We proceed by induction on $\ell$. 
	
It is trivial to see that the result holds for $\ell=1$. Assume that it is true for a general $\ell$, that is,
	\begin{align}\label{indHyp}
		A(q;\ell)=B(q;\ell)+C(q;\ell).
	\end{align}
	Our aim is to prove
	\begin{align}\label{indRes}
		A(q;\ell+1)=B(q;\ell+1)+C(q;\ell+1).
	\end{align}
We have
	\begin{align}\label{Aextra}
		A(q;\ell+1)=A(q;\ell)+\frac{(-q)_{n+\ell}(-1)_{n+\ell+1}(1+q^{2n+2\ell+2})q^{-(n+\ell+1)^2}}{(q)_{n+\ell+1}^2},
	\end{align}
	and
	\begin{align}\label{Cextra}
		C(q;\ell+1)=C(q;\ell)+\frac{(-1)^{\ell+1}(-q)_n(-1)_{n+\ell+1}q^{-n(n+\ell+1)}}{(q)_n(q)_{n+\ell+1}}.
	\end{align}
	Separating out the term corresponding to $j=n+\ell+1$ in $B(q;\ell+1)$, we see that
	\begin{align*}
		&B(q;\ell+1)-B(q;\ell)\\=&\frac{(-q)_{n+\ell+1}(-1)_{n+\ell+1}}{q^{(n+\ell+1)^2}(q)_{n+\ell+1}^2}+\frac{(-1)^{n+\ell+1}(-q)_{n+\ell}}{(q)_{n+\ell}} \sum_{j=n+1}^{n+\ell} \frac{(-1)_j(-1)^jq^{-j(n+\ell+1)}}{(q)_j}\bigg\{ \frac{(1+q^{n+\ell+1})}{(1-q^{n+\ell+1})}+q^j \bigg\}\\
		=&\frac{(-q)_{n+\ell+1}(-1)_{n+\ell+1}q^{-(n+\ell+1)^2}}{(q)_{n+\ell+1}^2}+\frac{(-1)^{n+\ell+1}(-q)_{n+\ell}}{(q)_{n+\ell+1}} \sum_{j=n+1}^{n+\ell} \frac{(-1)_{j+1}(-1)^jq^{-j(n+\ell+1)}}{(q)_j}\\&+\frac{(-1)^{n+\ell+1}(-q)_{n+\ell}}{(q)_{n+\ell+1}} \sum_{j=n+1}^{n+\ell} \frac{(-1)_j(-1)^jq^{-(j-1)(n+\ell+1)}}{(q)_{j-1}}\\
		=&\frac{(-q)_{n+\ell+1}(-1)_{n+\ell+1}}{q^{(n+\ell+1)^2}(q)_{n+\ell+1}^2}-\frac{(-q)_{n+\ell}(-1)_{n+\ell+1}q^{-(n+\ell)(n+\ell+1)}}{(q)_{n+\ell+1}(q)_{n+\ell}}+\frac{(-1)^{\ell}(-q)_{n+\ell}(-1)_{n+1}q^{-n(n+\ell+1)}}{(q)_{n+\ell+1}(q)_n}\\&+\frac{(-1)^{n+\ell+1}(-q)_{n+\ell}}{(q)_{n+\ell+1}}\bigg\{ \sum_{j=n+1}^{n+\ell-1} \frac{(-1)_{j+1}(-1)^jq^{-j(n+\ell+1)}}{(q)_j}+\sum_{j=n+2}^{n+\ell} \frac{(-1)_j(-1)^jq^{-(j-1)(n+\ell+1)}}{(q)_{j-1}} \bigg\}\\
		=&\frac{(-q)_{n+\ell}(-1)_{n+\ell+1}q^{-(n+\ell+1)^2}}{(q)_{n+\ell+1}(q)_{n+\ell}}\bigg\{\frac{(1+q^{n+\ell+1})}{(1-q^{n+\ell+1})}-q^{n+\ell+1}\bigg\}+\frac{(-1)^{\ell}(-1)_{n+\ell+1}(-q)_nq^{-n(n+\ell+1)}}{(q)_{n+\ell+1}(q)_n}\\=&\frac{(-q)_{n+\ell}(-1)_{n+\ell+1}(1+q^{2n+2\ell+2})q^{-(n+\ell+1)^2}}{(q)_{n+\ell+1}^2}+\frac{(-1)^{\ell}(-1)_{n+\ell+1}(-q)_nq^{-n(n+\ell+1)}}{(q)_{n+\ell+1}(q)_n}.
	\end{align*}
	Using \eqref{Aextra}, \eqref{Cextra} and the above equation, we see that
	\begin{align*}
		B(q;\ell+1)-B(q;\ell)=A(q;\ell+1)-A(q;\ell)+C(q;\ell)-C(q;\ell+1).
	\end{align*}
	An application of the induction hypothesis \eqref{indHyp} now leads to \eqref{indRes}.
	\end{proof}

\begin{proof}[Theorem \textup{\ref{theoremgfinite1}}][]
It is easy to see that the identity is trivial for $n<0$. So let $n\geq0$, and first assume $\ell\geq1$.  
Replace $n$ by $n+1$ in \eqref{finite1}, then add $\sum_{j=n+1}^{n+\ell}\frac{(-q)_{j-1}(1+q^{2j})(-1)_jq^{-j^2}}{(q)_j^2}$ to both sides of the resulting identity to obtain
\begin{align}\label{lfinite2}
	&1+\sum_{j=1}^{n+\ell} \frac{(-q)_{j-1}(1+q^{2j})(-1)_jq^{-j^2}}{(q)_j^2}\nonumber\\ &=\sum_{j=n+1}^{n+\ell} \frac{(-q)_{j-1}(1+q^{2j})(-1)_jq^{-j^2}}{(q)_j^2}+\frac{(-1)^n(-q)_n}{(q)_n} \sum_{j=0}^{n} \frac{(-1)_j(-1)^jq^{-jn}}{(q)_j}\nonumber\\ &=\sum_{j=n+1}^{n+\ell} \frac{(-q)_{j-1}(1+q^{2j})(-1)_jq^{-j^2}}{(q)_j^2}-\frac{(-1)^n(-q)_n}{(q)_n} \sum_{j=n+1}^{n+\ell} \frac{(-1)_j(-1)^jq^{-jn}}{(q)_j}\nonumber\\&\quad+\frac{(-1)^n(-q)_n}{(q)_n} \sum_{j=0}^{n+\ell} \frac{(-1)_j(-1)^jq^{-jn}}{(q)_j}.
\end{align}
Now replace $n$ by $n+\ell+1$ in (\ref{finite1}) to get
\begin{align}\label{lfinite3}
	&1+\sum_{j=1}^{n+\ell} \frac{(-q)_{j-1}(1+q^{2j})(-1)_jq^{-j^2}}{(q)_j^2}=\frac{(-1)^{n+\ell}(-q)_{n+\ell}}{(q)_{n+\ell}} \sum_{j=0}^{{n+\ell}} \frac{(-1)_j(-1)^jq^{-j(n+\ell)}}{(q)_j}\nonumber\\
	&=\frac{(-1)^{n+\ell}(-q)_{n+\ell}}{(q)_{n+\ell}} \sum_{j=n+1}^{{n+\ell}} \frac{(-1)_j(-1)^jq^{-j(n+\ell)}}{(q)_j}+\frac{(-1)^{n+\ell}(-q)_{n+\ell}}{(q)_{n+\ell}} \sum_{j=0}^{n} \frac{(-1)_j(-1)^jq^{-j(n+\ell)}}{(q)_j}.
\end{align}
Now compare (\ref{lfinite2}) and (\ref{lfinite3}), then use Lemma \ref{lfinitelemma} to get the required result for $\ell\geq1$.

We now prove the identity for the remaining values of $\ell$. Observe that the case $\ell=0$ is trivial. 

Now let $-n\leq\ell\leq-1$. We let $\ell=-m$, where $1\leq m\leq n$. Then the identity to be proved is
 \begin{align*}
 	\sum_{j=0}^{n} \frac{(-1)_j(-1)^jq^{-j(n-m)}}{(q)_j}=(-1)^{\ell}\frac{(q^{n+1})_{-m}}{(-q^{n+1})_{-m}} \sum_{j=0}^{n-m} \frac{(-1)_j(-1)^jq^{-jn}}{(q)_j} .
 \end{align*}
Now replace $n$ by $n+m$ and use the fact that $\frac{(q^{n+m+1})_{-m}}{(-q^{n+m+1})_{-m}}=\frac{(-q^{n+1})_m}{(q^{n+1})_m}$
 to see that what one needs to prove is nothing but \eqref{general ell finite identity} for $\ell=m\geq1$. But this is already proved above.
 
 Finally, for $\ell<-n$, the identity is seen to be true once we multiply both sides by $\frac{(-q^{n+1})_{\ell}}{(q^{n+1})_{\ell}}$ and observe that the resulting identity is trivial.
\end{proof}

\section{A congruence involving $b(n)$, $b_{1}(n)$ and coefficients of tenth order mock theta functions}\label{tenth}

A generalized modular relation found on page $27$ of Ramanujan's Lost Notebook \cite{lnb} states that for $a\in\mathbb{C}\backslash\{0\}$, and $b\in\mathbb{C}$,
\begin{align}\label{theta_rrsum}
	&\sum_{m=0}^{\infty}\frac{a^{-2m}q^{m^2}}{(bq)_m}\sum_{n=0}^{\infty}\frac{a^nb^nq^{{n^2}/4}}{(q)_n}+\sum_{m=0}^{\infty}\frac{a^{-2m-1}q^{m^2+m}}{(bq)_m}\sum_{n=0}^{\infty}\frac{a^nb^nq^{{(n+1)^2}/4}}{(q)_n}\nonumber\\
	&=\frac{1}{(bq)_{\infty}}\sum_{n=-\infty}^{\infty}a^nq^{n^2/4}-(1-b)\sum_{n=1}^{\infty}a^nq^{n^2/4}\sum_{\ell=0}^{n-1}\frac{b^\ell}{(q)_\ell}.
\end{align}
Andrews \cite{yesto} obtained the first proof of the above identity. The first and the second authors \cite{adixit gkumar} have recently generalized Ramanujan's identity using an extra parameter $s\in\mathbb{N}$, and have shown that their generalization reduces to \eqref{theta_rrsum} for $s=2$.

In what follows, we show that using Ramanujan's identity \eqref{theta_rrsum}, we can obtain a relation between the generating functions of $b_\ell(n)$ and $b_{\ell+1}(n)$ for any $\ell\in\mathbb{N}$.
\begin{theorem}
	For $\ell\in\mathbb{N}\cup\{0\}$,
	\begin{align}\label{6.2}
		&\sum_{m=0}^{\infty}\frac{(-1)^mq^{m^2+m\ell}}{(-q)_m}\sum_{n=0}^{\infty}\frac{(-1)^nq^{n^2-n\ell}}{(q)_{2n}}-\sum_{m=0}^{\infty}\frac{(-1)^mq^{m^2+m+m\ell}}{(-q)_m}\sum_{n=0}^{\infty}\frac{(-1)^nq^{(n+1)^2-n\ell}}{(q)_{2n+1}}\nonumber\\
		&=\frac{1}{(-q)_{\infty}}\sum_{n=-\infty}^{\infty}(-1)^nq^{n^2-n\ell}-2\sum_{n=1}^{\infty}(-1)^nq^{n^2-n\ell}\sum_{k=0}^{2n-1}\frac{(-1)^k}{(q)_k}.
	\end{align}
	\begin{align}\label{6.3}
		&\sum_{m=0}^{\infty}\frac{(-1)^mq^{m^2+m\ell}}{(-q)_m}\sum_{n=0}^{\infty}\frac{(-1)^nq^{n^2+n-n\ell}}{(q)_{2n+1}}+\sum_{m=0}^{\infty}\frac{(-1)^mq^{m^2+m+m\ell}}{(-q)_m}\sum_{n=0}^{\infty}\frac{(-1)^nq^{n^2+n-n\ell+\ell/2}}{(q)_{2n}}\nonumber\\
		&=\frac{-1}{(-q)_{\infty}}\sum_{n=-\infty}^{\infty}(-1)^nq^{n^2+n-n\ell}+2\sum_{n=0}^{\infty}(-1)^nq^{n^2+n-n\ell}\sum_{k=0}^{2n}\frac{(-1)^k}{(q)_k}.
	\end{align}
\end{theorem}
\begin{proof}
We prove the results for $-1<q<1$. They can be extended for $|q|<1$ by the principle of analytic continuation.

 Let $a=iq^{-\ell/2}$ and $b=-1$ in \eqref{theta_rrsum}. This gives
		\begin{align}\label{thmRLN2gen}
		&\sum_{m=0}^{\infty}\frac{(-1)^mq^{m^2+m\ell}}{(-q)_m}\sum_{n=0}^{\infty}\frac{(-i)^nq^{{n^2}/4-n\ell/2}}{(q)_n}-iq^{\frac{\ell}{2}}\sum_{m=0}^{\infty}\frac{(-1)^mq^{m^2+m(\ell+1)}}{(-q)_m}\sum_{n=0}^{\infty}\frac{(-i)^nq^{{(n+1)^2}/4-n\ell/2}}{(q)_n}\nonumber\\
		&=\frac{1}{(-q)_{\infty}}\sum_{n=-\infty}^{\infty}i^nq^{n^2/4-n\ell /2}-2\sum_{n=1}^{\infty}i^nq^{n^2/4- n\ell/2}\sum_{k=0}^{n-1}\frac{(-1)^k}{(q)_k}.
	\end{align}
	Now split the two sums over $n$ on the left-hand side based on its parity, and then group the real terms together and imaginary terms together so as to get
	\begin{align}\label{RLN_LHS}
		&\sum_{m=0}^{\infty}\frac{(-1)^mq^{m^2+m\ell}}{(-q)_m}\sum_{n=0}^{\infty}\frac{(-i)^nq^{{n^2}/4-n\ell/2}}{(q)_n}-iq^{\ell/2}\sum_{m=0}^{\infty}\frac{(-1)^mq^{m^2+m+m\ell}}{(-q)_m}\sum_{n=0}^{\infty}\frac{(-i)^nq^{{(n+1)^2/4-n\ell/2}}}{(q)_n}\nonumber\\
		&=\sum_{m=0}^{\infty}\frac{(-1)^mq^{m^2+m\ell}}{(-q)_m}\sum_{n=0}^{\infty}\frac{(-1)^nq^{n^2-n\ell}}{(q)_{2n}}-\sum_{m=0}^{\infty}\frac{(-1)^mq^{m^2+m+m\ell}}{(-q)_m}\sum_{n=0}^{\infty}\frac{(-1)^nq^{(n+1)^2-n\ell}}{(q)_{2n+1}}\nonumber\\&\ -iq^{1/4}\bigg\{\sum_{m=0}^{\infty}\frac{(-1)^mq^{m^2+m\ell}}{(-q)_m}\sum_{n=0}^{\infty}\frac{(-1)^nq^{n^2+n-n\ell-\ell/2}}{(q)_{2n+1}}+q^{\frac{\ell}{2}}\sum_{m=0}^{\infty}\frac{(-1)^mq^{m^2+m+m\ell}}{(-q)_m}\sum_{n=0}^{\infty}\frac{(-1)^nq^{n^2+n-n\ell}}{(q)_{2n}}\bigg\}.
	\end{align}
	Similarly, split the two sums over $n$ on the right-hand side of \eqref{thmRLN2gen} based on its parity to get
		\begin{align}\label{RLN_RHS}
		&\frac{1}{(-q)_{\infty}}\sum_{n=-\infty}^{\infty}i^nq^{n^2/4-n\ell /2}-2\sum_{n=1}^{\infty}i^nq^{n^2/4- n\ell/2}\sum_{k=0}^{n-1}\frac{(-1)^k}{(q)_k}\nonumber\\
		&=\frac{1}{(-q)_{\infty}}\sum_{n=-\infty}^{\infty}(-1)^nq^{n^2-n\ell}-2\sum_{n=1}^{\infty}(-1)^nq^{n^2-n\ell}\sum_{k=0}^{2n-1}\frac{(-1)^k}{(q)_k}\nonumber\\
		& -iq^{1/4}\bigg\{\frac{-1}{(-q)_{\infty}}\sum_{n=-\infty}^{\infty}(-1)^nq^{n^2+n-n\ell}+2\sum_{n=0}^{\infty}(-1)^nq^{n^2+n-n\ell}\sum_{k=0}^{2n}\frac{(-1)^k}{(q)_k}\bigg\}.
	\end{align}
Now substitute the left-hand side of \eqref{thmRLN2gen} by the right-hand side of \eqref{RLN_LHS}, and likewise, the right-hand side of \eqref{thmRLN2gen} by the right-hand side of \eqref{RLN_RHS}.  The identities in \eqref{6.2} and \eqref{6.3} then follow upon equating the real and imaginary parts on both sides of the resulting identity. 
	\end{proof}
\begin{corollary}
	\begin{align}\label{1st}
		&\sum_{m=0}^{\infty}\frac{(-1)^mq^{m^2}}{(-q)_m}\sum_{n=0}^{\infty}\frac{(-1)^nq^{n^2}}{(q)_{2n}}-\sum_{m=0}^{\infty}\frac{(-1)^mq^{m^2+m}}{(-q)_m}\sum_{n=0}^{\infty}\frac{(-1)^nq^{(n+1)^2}}{(q)_{2n+1}}\nonumber\\
		&=\frac{1}{(-q)_{\infty}}\sum_{n=-\infty}^{\infty}(-1)^nq^{n^2}-2\sum_{n=1}^{\infty}(-1)^nq^{n^2}\sum_{k=0}^{2n-1}\frac{(-1)^k}{(q)_k}.
	\end{align}
	\begin{align*}
		&\sum_{m=0}^{\infty}\frac{(-1)^mq^{m^2}}{(-q)_m}\sum_{n=0}^{\infty}\frac{(-1)^nq^{n^2+n}}{(q)_{2n+1}}+\sum_{m=0}^{\infty}\frac{(-1)^mq^{m^2+m}}{(-q)_m}\sum_{n=0}^{\infty}\frac{(-1)^nq^{n^2+n}}{(q)_{2n}}\nonumber\\
		&=\frac{-1}{(-q)_{\infty}}\sum_{n=-\infty}^{\infty}(-1)^nq^{n^2+n}+2\sum_{n=0}^{\infty}(-1)^nq^{n^2+n}\sum_{k=0}^{2n}\frac{(-1)^k}{(q)_k}.
	\end{align*}
\end{corollary}
\begin{proof}
Let $\ell=0$ in \eqref{6.2} and \eqref{6.3}.
\end{proof}
\begin{remark}
	The sums $\sum_{n=0}^{\infty}\frac{(-1)^nq^{n^2}}{(q)_{2n}}$ and $\sum_{n=0}^{\infty}\frac{(-1)^nq^{(n+1)^2}}{(q)_{2n+1}}$ occurring on the left-hand side of \eqref{1st} remind us of the tenth order mock theta functions $X(q):=\sum_{n=0}^{\infty}\frac{(-1)^nq^{n^2}}{(-q)_{2n}}$ and  $\chi(q):=\sum_{n=0}^{\infty}\frac{(-1)^nq^{(n+1)^2}}{(-q)_{2n+1}}$.
\end{remark}
Indeed, the observation in the above remark helps us to obtain the following result.
\begin{theorem}
Define
	\begin{align*}
	X(q)&:= \sum_{n=1}^{\infty}\frac{(-1)^{n}q^{n^2}}{(-q)_{2n}}=: \sum_{m=0}^{\infty}c_{X}(m)q^{m},\\
	\chi(q)&:= \sum_{n=1}^{\infty}\frac{(-1)^{n}q^{(n+1)^2}}{(-q)_{2n+1}}=:\sum_{m=0}^{\infty}c_{\chi}(m)q^{m}.
	\end{align*}
	Then, for $k \in \mathbb{N}$,
\begin{align*}
	\sum_{n=0}^{k}b(n)c_X(k-n)+b_1(n)c_{\chi}(k-n) \equiv0\pmod{2}.
	\end{align*}
\end{theorem}
\begin{proof}
	First note that
	$ \sum_{n=0}^{\infty}\frac{(-1)^n q^{n^2}}{(q)_{2n}} \equiv X(q) \pmod{2}$ and 	$ \sum_{n=0}^{\infty}\frac{(-1)^n q^{(n+1)^2}}{(q)_{2n+1}} \equiv \chi(q) \pmod{2}$.  
	Now multiplying both sides of \eqref{1st} by $(-q)_\infty$, then reducing the resulting identity modulo $2$ and then using the above congruences, we get
	\begin{align*}
		(-q)_{\infty}\sigma_2(q) X(q)-(-q)_{\infty}\sigma_{2,1}(q) \chi(q) \equiv  \sum_{n=-\infty}^{\infty}(-1)^n q^{n^2} \pmod{2} \equiv1 \pmod{2}.
	\end{align*}	
The result now follows from the fact that
	\begin{align*}
		(-q)_{\infty}\sigma_2(q) X(q)-(-q)_{\infty}\sigma_{2,1}(q) \chi(q) =  \sum_{k=0}^{\infty}\sum_{n=0}^{k}(b(n)c_X(k-n)-b_1(n)c_{\chi}(k-n) ) q^k.
	\end{align*} 
\end{proof}
\begin{remark}
From \cite[Theorem 1.3]{matsusaka}, we have $c_5(40n-1)\equiv c_{X_{10}}(n)\equiv\pmod{2}$ and $c_5(40n-9)\equiv c_{\chi_{10}}(n)\equiv\pmod{2}$,
where $c_5(n)$ is the Fourier coefficient of the Hauptmoduln $j_5(\tau)$, that is, of
\begin{align*}
j_5(\tau):=\frac{1}{q}\frac{(q;q)_{\infty}^6}{(q^5;q^5)_{\infty}^6}=\frac{1}{q}-6+\sum_{n=1}^{\infty}c_5(n)q^n.
\end{align*}
 Hence, we have
	\begin{align*}
	\sum_{n=0}^{k}b(n)c_5(40k-40n-1)+b_1(n)c_5(40k-40n-9)\equiv0\pmod{2}.
\end{align*}
\end{remark}

\section{Unrestricted Rascoe and non-Rascoe partitions}\label{analogues}
As alluded to in the introduction, we now remove, from the definition of Rascoe and non-Rascoe partitions, the restriction that the parts of such a partition be distinct, that is, we consider the \emph{unrestricted Rascoe and non-Rascoe partitions}. Let $c(n)$ denote the partitions of $n$ in which the number of parts is a part, and let $e(n)$ denote the partitions of $n$ in which the number of parts is not a part. We now obtain two different representations for their generating functions. Out of these, only the second representation of the generating function of $e(n)$ derived in the theorem below is known, namely, in  \cite{jtr2a}, where it is given without proof.
\begin{theorem}\label{rascoe analogue}
	\begin{align}
		\sum_{N=1}^{\infty} c(N) q^N 
		&= q+\sum_{n=2}^{\infty}  \sum_{m=0}^{n} \qbinom{2n-m-2}{n-m-1}\frac{q^{mn+2n-1}}{(q)_{m}}=\frac{q}{(q^2)_{\infty}},\label{ra1}\\
		\sum_{N=1}^{\infty} e(N) q^N	&= 1+\frac{q^2}{1-q}+\sum_{n=2}^{\infty}  \sum_{m=0}^{n} \qbinom{2n-m-2}{n-m}\frac{q^{mn+n}}{(q)_{m}}=\frac{1-q+q^2}{(q)_{\infty}}.\label{ra2}
	\end{align}
\end{theorem}
\begin{proof}
We begin by showing that the double sum in \eqref{ra1} generates $c(N)$. To that effect, let $n$ denote the number of parts of a partition $\pi$ of $N$  and $m$ denote the number of parts greater than $n$. Note that $n$ itself is a part of $\pi$. We initially assume $n\geq2$. 

The Ferrers diagram of $\pi$ consists of a rectangle of size $m\times(n+1)$ lying just above the first occurrence of $n$, and a partition next to the right of the rectangle, generated by $1/(q)_m$.

Moreover, the portion lying below the first occurrence of $n$ constitutes a sub-partition into exactly $n-m-1$ parts each of which is less than or equal to $n$, which is generated by \cite[p.~34]{gea} $q^{n-m-1}\qbinom{2n-m-2}{n-m-1}$. Therefore, part of the complete generating function we get this way is given by
\begin{equation}\label{part}
\sum_{n=2}^{\infty}  \sum_{m=0}^{n} \qbinom{2n-m-2}{n-m-1}\frac{q^{mn+2n-1}}{(q)_{m}}.
\end{equation}
The case $n=0$, corresponding to the empty partition, is not a part of the generating function for $c(N)$, and is considered, instead, in the one for $e(N)$. Moreover, the only partition corresponding to $n=1$ is $1$ itself. Therefore, its contribution towards the generating function of $c(N)$ is $q$. Hence, adding it to \eqref{part} gives the first equality in \eqref{ra1}. 

To derive the equality between the extreme sides of \eqref{ra1}, we show that the partitions enumerated by $c(N)$ can be bijectively mapped to the set of partitions in which $1$ appears exactly once. To see this, let $k$ be the number of parts in a partition $\pi$ enumerated by $c(N)$. Then there is at least one part in $\pi$ with $k$ nodes in its Ferrers diagram. We remove the $k$ nodes corresponding to the first occurrence of this part and distribute them, one each, to the remaining $k-1$ parts, with an additional one node conceived as the smallest part of the modified partition. This, clearly, is a partition enumerated by $q/(q^2)_{\infty}$. Since the process is reversible, we get a bijection, which proves that $c(N)$ is generated by $q/(q^2)_{\infty}$. This completes the proof of \eqref{ra1}. 

The proof of first equality of \eqref{ra2} is along the similar lines. The only thing to note is that if a partition generated by $e(N)$ has a single part, say $k$, then $k>1$. Such partitions are generated by $q^2/(1-q)$. Finally, the equality between the extreme sides of \eqref{ra2} follows from the fact that $c(N)+e(N)=p(N)$ and the equality between the extreme sides of \eqref{ra1}.
\end{proof}

\begin{remark}
George Beck \cite{jtr2a} has conjectured that $e(n)$ is the total number of distinct parts of each partition of $2n+2$ with rank $n+1$.
\end{remark}
\section{Concluding Remarks}\label{cr} 
The study of $\sigma_2(q)$ undertaken in this paper arose by asking ourselves a natural question - \emph{if the generating function of the excess number of partitions into distinct parts with even rank over those with odd rank, that is, $\sigma(q)$, has rich properties from the point of view of partition theory, algebraic number theory, Maass wave forms as well as quantum modular forms, it might be interesting to investigate a similar function starting with partitions in which the difference between any two parts is at least $2$.}

As demonstrated in this paper, $\sigma_2(q)$ does satisfy beautiful properties. For example, it is intimately connected to an interesting class of partition functions, namely, the \emph{non-Rascoe partitions}, and this connection not only enables us to precisely understand the parity of $b(n)$, the number of non-Rascoe partitions of $n$, but also help represent $b(n)$ as a sum involving the number of partitions where only the smallest part is allowed to repeat and its frequency is the smallest part itself.

A result analogous to Theorem \ref{theorem main}, but for $\sigma(q)$, was obtained very recently in \cite[Theorem 1.1]{donato}, namely,
\begin{align}\label{sigma mex gf}
\frac{(-q)_{\infty}}{(q)_{\infty}}\sigma(q)=\sum_{n=0}^{\infty}\sigma\widehat{\textup{mex}}(n)q^n,
\end{align}
where $\sigma\widehat{\textup{mex}}(n):=\sum_{\pi\in\overline{P}(n)}\widehat{\textup{mex}}(n)$ with $\widehat{\textup{mex}}(\pi)$ denoting the minimal excludant of the overlined parts of an overpartition $\pi$.	(Overpartitions are partitions in which the first (or equivalently the final) occurrence of a part may or may not be overlined.)

While we have initiated the study of $\sigma_2(q)$ as well as of $b(n)$, a lot of questions still remain to be seen. We have collected them below. Each of these problems is worth exploring.  

\begin{enumerate}
	\item In his Lost Notebook, Ramanujan gave the following new representation of $\sigma(q)$ (see \cite[Equation (1.6)]{andrews1986} for a proof):
	\begin{equation}\label{sigma another}
	\sigma(q)=1+q\sum_{n=0}^{\infty}(q)_n(-q)^n.
	\end{equation}
	Our first problem asks whether there exists an analogous representation for $\sigma_2(q)$. Our attempts toward finding such a representation did not work. The importance of having it, if at all it exists, is highlighted in (2) below.
	
	\item Find modular properties of $\sigma_2(q)$. As shown by Zagier \cite[p.~663]{zagierqmf}, the function $q^{1/24}\sigma(q)$ is one of the first examples of what he calls a \emph{quantum modular form}. This establishes the modularity of $\sigma(q)$. What kind of modular behavior does $\sigma_2(q)$ exhibit?
	
	 Cohen established an analogue of \eqref{sigma another} for $\sigma^{*}(q)$, namely,
	$\sigma^{*}(q)=-2\sum_{n=0}^{\infty}q^{n+1}(q^2;q^2)_n$, and then discovered the beautiful relation $\sigma(q)=-\sigma^{*}(q^{-1})$, which holds whenever $q$ is a root of unity.  
	
	Zagier \cite[p.~664]{zagierqmf} used this relation to show that if $q=e^{-t}$, then as $t\to0^{+}$,
	\begin{equation*}
	\sigma(q)=2-2t+5t^2-\frac{55}{3}t^3+\frac{1073}{12}t^4-\frac{32671}{60}t^5+\frac{286333}{72}t^6-\cdots,
	\end{equation*} 
and Doneto \cite[Theorem 1.3]{donato} used this result along with \eqref{sigma mex gf} to show that $\sigma\widehat{\textup{mex}}(n)\sim e^{\pi\sqrt{n}}/(4n)$ as $n\to\infty$. 

Thus, if there exists a representation for $\sigma_2(q)$ analogous to \eqref{sigma another} and one is able to proceed in a similar manner as that for $\sigma(q)$, it may allow us to find an asymptotic formula of $b(n)$ as $n\to\infty$.

After shifting from unrestricted partitions to partitions into distinct parts, one observes a stark difference in the modular nature of the concerned objects. For example, in the first case, the rank parity function is the third order mock theta function $f(q)$ whereas in the second, it is the quantum modular form $\sigma(q)$. Similarly, while the generating function of unrestricted non-Rascoe partitions, considered in Section \ref{analogues}, is simply a linear combination of simple modular objects, the generating function of non-Rascoe partitions themselves is $(-q)_\infty\sigma_2(q)$ whose modularity is yet to be understood.

\item If $p_d(n)$ denotes the number of partitions of $n$ into distinct parts, then Theorems \ref{theoremrrp} and \ref{theorem main} imply that
\begin{equation*}
b(n)=\sum_{j=0}^{n}\sum_{m=-\infty}^{\infty}(-1)^{m-1}R(m,j)p_d(n-j).
\end{equation*}
Our proof of Theorem \ref{theorem main} involves a mix of analytic and combinatorial methods. It would be of interest to see if the above equation could be proved directly.

\item It seems that the arithmetic progression $29k+21$ in our mod $4$ congruence conjecture appears from nowhere. Probably its appearance can be better understood once the modularity of $\sigma_2(q)$ is established. If true, what are the implications of such a congruence? We weren't able to make sense of why the exceptions to the condition $k\neq29m$ given in Remark \ref{exceptions} exist. Do the values of $m$ given there hint at something important?

Our search for congruence satisfied  by $b_1(n)$ did not yield anything. It seems highly unlikely that there exists a congruence for $b_\ell(n)$ for higher $\ell$. 

On a side note, coefficients of several mock theta functions over certain arithmetic progressions obey mod $4$ congruences; see, for example, \cite{brietzke}, \cite{chan-mao}, \cite{chen-garvan}, \cite{hu-liu-yao} and \cite{mao}.
\item Can $\sigma_2(q)$ be represented as a Hecke-Rogers series? Theorem \ref{hecke-type} hints at such a possibility, however, we weren't able to further simplify our representations. If not in general, can the expressions in Theorem \ref{hecke-type} be reduced to simpler objects at least modulo $4$?  There exist precedents where complicated objects reduced to simpler ones modulo $4$. 

For example, Andrews, Schultz, Yee and the first author \cite[Theorem 1.3]{adsy1} showed that the following series involving the little $q$-Jacobi polynomials which arises in a representation of the generating function of $\overline{p}_\omega(n)$, that is, the number of overpartitions of $n$ in which all odd parts are less than twice the smallest parts, obeys the congruence
\begin{align*}
\sum_{n=0}^{\infty}\frac{(q;q^2)_n(-q)^n}{(-q;q^2)_{n}(1+q^{2n})}\sum_{j=0}^{2n}\frac{(-1;q)_j}{(q;q)_j}(-q)^j\equiv\frac{1}{2} \frac{(q;q^2)_{\infty}}{(-q;q^2)_{\infty}} \pmod{4}.
\end{align*}
However, mimicking its simpler proof given in \cite{dixit-ac} for either of the representations in Theorem \ref{hecke-type}, unfortunately, did not work. 
\item Give $q$-series proofs of the results in Theorem \ref{rascoe analogue}.
\end{enumerate}

\begin{center}
\textbf{Acknowledgements}
\end{center}
Our initial \emph{Mathematica} code gave first $30000$ values in the sequence $\{b(n)\}_{n=1}^{\infty}$ to verify Conjecture 1, but thanks to M. S. A. Alam Khan of IIT Gandhinagar for his efficient \emph{Mathematica} code, which helped us verify our conjecture up $n\leq10^5$. The first author is supported by the Swarnajayanti Fellowship grant SB/SJF/2021-22/08 of ANRF (Government of India). The third author is supported by NBHM-DAE Fellowship (Government of India). Both the authors sincerely thank the respective funding agencies for their support.

\end{document}